\documentclass[11pt,a4paper]{amsart}

\usepackage{amssymb, amstext, amscd, amsmath, color}

\usepackage{url}
\usepackage{tikz}
\usepackage{epstopdf}
\usepackage{amsfonts}
\usepackage{amsthm}
\usepackage{amsfonts}
\usepackage{array}
\usepackage{epsfig}
\usepackage{eucal}
\usepackage{latexsym}
\usepackage{mathrsfs}
\usepackage{textcomp}
\usepackage{verbatim}
\usepackage{setspace}

\newtheorem{thm}{Theorem}[section]

\newtheorem{cor}[thm]{Corollary}

\newtheorem{lem}[thm]{Lemma}

\newtheorem{rem}[thm]{Remark}

\numberwithin{equation}{section}

\newcommand{\bS}{{\mathbb{S}}}

  \newcommand{\U}{{\mathcal{U}}}

\newcommand{\be}{{\bf{e}}}

\usepackage{blkarray}

\begin{document}

\title[Pairing symmetries for Euclidean and spherical frameworks]{Pairing symmetries for Euclidean and spherical frameworks}
\author[K. Clinch]{K. Clinch}
\address{Dept.\ Math.\ Informatics \\ University of Tokyo\\
Tokyo \\Japan}
\email{clinch@mist.i.u-tokyo.ac.jp}
\author[A. Nixon]{A. Nixon}
\address{Dept.\ Math.\ Stats.\\ Lancaster University\\
Lancaster LA1 4YF \\U.K.}
\email{a.nixon@lancaster.ac.uk}
\author[B. Schulze]{B. Schulze}
\address{Dept.\ Math.\ Stats.\\ Lancaster University\\
Lancaster LA1 4YF \\U.K.}
\email{b.schulze@lancaster.ac.uk}
\author[W. Whiteley]{W. Whiteley}
\address{Dept.\ Math.\ Stats.\\ York University\\
Toronto  \\Canada}
\email{whiteley@mathstat.yorku.ca}

\begin{abstract}
In this paper we consider the effect of symmetry on the rigidity of bar-joint frameworks, spherical frameworks and point-hyperplane frameworks in $\mathbb{R}^d$. In particular we show that, under forced or incidental symmetry, infinitesimal rigidity for 
spherical frameworks with vertices in $X$ on the equator and point-hyperplane frameworks with the vertices in $X$ representing hyperplanes are equivalent.
We then show, again under forced or incidental symmetry, that infinitesimal rigidity properties under certain symmetry groups can be paired, or clustered, under inversion on the sphere so that infinitesimal rigidity with a given group is equivalent to infinitesimal rigidity under a paired group. The fundamental basic example is that mirror symmetric rigidity is equivalent to half-turn symmetric rigidity on the 2-sphere. With these results in hand we also deduce some combinatorial consequences for the rigidity of symmetric bar-joint and point-line frameworks.
\end{abstract}

\date{}
\maketitle

\section{Introduction}

Given a collection of primitive geometric objects in a space satisfying particular geometric constraints, a fundamental question is whether the given constraints uniquely determine the whole configuration up to congruence. The rigidity problem for bar-joint frameworks in $\mathbb{R}^d$, where the objects are points, the constraints are pairwise distances and only local deformations are considered, is a classical example. Mathematically, a  (bar-joint) framework  in $\mathbb{R}^d$  is defined to be a pair $(G,p)$, consisting of an undirected finite graph $G=(V,E)$ and a map $p:V\to \mathbb{R}^d$.

A framework $(G,p)$ in $\mathbb{R}^d$ is \emph{rigid} if the only edge-length-preserving continuous motions of the vertices arise from isometries of $\mathbb{R}^d$. In general, when $d\geq 2$, it is NP-hard to determine if a given framework is rigid \cite{Abb}.

A standard approach to study the rigidity of bar-joint frameworks is to linearise the problem by differentiating the length constraints on the corresponding pairs of points. This leads to the notion of infinitesimal (or equivalently, static) rigidity.
An \emph{infinitesimal motion} of a framework $(G,p)$  in $\mathbb{R}^d$
is a function $u: V\to \mathbb{R}^{d}$ such that
\begin{equation}
\label{infinmotioneq}
\langle p_i-p_j, u_i-u_j\rangle =0 \quad\textrm{ for all } \{i,j\} \in  E\textrm{,}
\end{equation}
where $p_i=p(i)$ and $u_i=u(i)$ for each $i$. An infinitesimal motion $u$ of $(G,p)$ is a \emph{trivial infinitesimal motion}
if there exists a skew-symmetric matrix $S$
and a vector $t$ such that $u_i=S p_i+t$ for all $i\in V$.
$(G,p)$ is \emph{infinitesimally rigid} if every infinitesimal motion of $(G,p)$ is trivial, and \emph{infinitesimally flexible} otherwise. Moreover if the framework is suitably generic then rigidity and infinitesimal rigidity coincide \cite{AR}.

Pogorelov~\cite[Chapter V]{Pog} observed that the space of infinitesimal motions of a bar-joint framework that is constrained to lie on a strict semi-sphere is isomorphic to those of the framework obtained by a central projection to Euclidean space. Since then, connections between various types of rigidity models in different spaces have been extensively studied, see, e.g.,~\cite{Izmestiev,SaliolaWh,BSWWSphere}. When talking about  infinitesimal rigidity, these connections are often just consequences of the fact that  infinitesimal rigidity is preserved by projective transformations \cite{CrW,Ran}. A key essence of the research is its geometric and combinatorial interpretations, which sometimes give us unexpected connections between theory and real applications.

In \cite{EJNSTW} this line of research was extended to include \emph{point-hyperplane frameworks} which consist of points and hyperplanes combined with point-point distance constraints, point-hyperplane distance constraints and hyperplane-hyperplane angle constraints (see Section \ref{subsec:ph} for a rigorous definition). These types of frameworks have practical applications in areas such as mechanical and civil  engineering as well as CAD, since point-hyperplane distance constraints may be used to model slider-joints in engineering structures \cite{EJNSTW,JO}. In particular the following result showed that the (infinitesimal) rigidity of such frameworks is equivalent to the (infinitesimal) rigidity of Euclidean and spherical frameworks with a certain special subset of vertices (that correspond to the hyperplanes). See Section~\ref{subsec:sph} for a detailed discussion of spherical frameworks.

\begin{thm}\label{thm:regulartransfer}\cite[Theorems 2.4 and 2.5]{EJNSTW}
Let $G=(V,E)$ be a graph and $X\subseteq V$. Then the following are equivalent:
\begin{itemize}
\item[(a)] $G$ can be realised as an infinitesimally rigid bar-joint framework on $\mathbb{S}^d$ such that the points assigned to $X$ lie on the equator.
\item[(b)] $G$ can be realised as an infinitesimally rigid point-hyperplane framework in $\mathbb{R}^d$ such that each vertex in $X$ is realised as a hyperplane and each vertex in $V\setminus X$ is realised as a point.
\item[(c)]  $G$ can be realised as an infinitesimally rigid bar-joint framework in $\mathbb{R}^d$ such that the points assigned to $X$ lie on a hyperplane.
\end{itemize}
\end{thm}

Symmetry plays a key role in some prominent applications of rigidity, such as the dynamics of proteins or the design of engineering structures, and the effect of symmetry on bar-joint frameworks has been well studied over the last decade \cite{jkt,mt1,OP,ST,SWorbit} (see also \cite{BShb,BSWWhb} for recent summaries of results). Note that there are two versions of symmetric rigidity: \emph{incidental} symmetry where a given framework is symmetric (and hence not `generic') but any continuous, or infinitesimal, motion is allowed; and \emph{forced} symmetry where a given framework is symmetric and it is considered to be rigid if the only possible motions destroy the symmetry. (Background definitions on symmetric frameworks are given in Section \ref{sec:backsym}.) 

In this paper we extend Theorem \ref{thm:regulartransfer} to symmetric frameworks. 
In particular given a framework that admits some point group symmetry we show, in Sections \ref{sec:transfer} and \ref{sec:forced}, that both forced-symmetric and incidentally symmetric infinitesimal rigidity can be transferred between spherical frameworks with a given set $X$ of vertices realised on the equator and point-hyperplane frameworks, where the vertices of $X$ are exactly the vertices realised as hyperplanes. 
We can give a full analogue of the theorem (i.e. showing a symmetric version of (c) is also equivalent) only in the case of mirror symmetry, again in both the forced and incidental cases.

It turns out that the impact of symmetry under the projective operations used to prove the above results reveal further unexpected equivalences. That is, certain pairs of symmetry groups turn out to provide identical infinitesimal rigidity properties. A fundamental example being that half-turn rotation and mirror symmetry on the 2-sphere have geometrically equivalent infinitesimal rigidity properties, in both the incidental and forced contexts. We give a detailed analysis of all such pairings on the 2-sphere in Section \ref{sec:pairs}, consider groups of involutions in higher dimensions in Section \ref{sec:higher} and discuss some consequences of these pairings, particularly from the combinatorial perspective, as we go.

Finally, in Section \ref{sec:non-free}, we consider the corresponding results when the action of the symmetry group is not free on the vertices of the symmetric graph. In this context we present some examples and again discuss some combinatorial consequences. In particular, we obtain a combinatorial characterisation of a special class of minimally infinitesimally rigid point-line frameworks with reflection symmetry. We conclude Section \ref{sec:non-free} with some observations on the projective/elliptical model which, via statics, is the root of the projective understanding of rigidity and connects to the projective basis of the pairings \cite{Con}.  

\medskip

\section{Rigidity of symmetric frameworks}
\label{sec:backsym}

\subsection{Symmetric graphs}

Let $G=(V,E)$ be a graph. An {\em automorphism} of $G$ is a permutation $\pi:V\rightarrow V$ such that $\{i,j\}\in E$ if and only if $\{\pi(i),\pi(j)\}\in E$.  The   group  of all automorphisms of  $G$ is denoted by $\textrm{Aut}(G)$. For an abstract group $\Gamma$, we say that $G$ is {\em $\Gamma$-symmetric} if there exists a group action  $\theta:\Gamma \to \textrm{Aut}(G)$.
 For the following definitions, we will assume that the action $\theta$ is free on the vertex set of $G$, and we will  omit $\theta$ if it is clear from the context. We will then simply write $\gamma i$ instead of $\theta(\gamma)(i)$.  

The {\em quotient graph} of a $\Gamma$-symmetric graph $G$ is the multigraph $G/\Gamma$ whose vertex set is the set $V/\Gamma$
of vertex orbits and whose edge set is the set $E/\Gamma$ of edge orbits. 
Note that an edge orbit may be represented by a loop in $G/\Gamma$. 
The \emph{(quotient) $\Gamma$-gain graph} of a $\Gamma$-symmetric graph $G$ is the pair $(G_0,\psi)$, where  $G_0=(V_0,E_0)$ is the quotient graph of $G$ with an orientation on the edges, and $\psi:E_0\to \Gamma$ is defined as follows. 
Each edge orbit $\Gamma e$ connecting $\Gamma i$ and $\Gamma j$ in $G/\Gamma$ can be written as $\{\{\gamma i,\gamma \circ\alpha j\}\mid \gamma\in \Gamma \}$ for a unique $\alpha\in \Gamma$. For each $\Gamma e$, orient $\Gamma e$ from $\Gamma i$ to $\Gamma j$ in $G/\Gamma$ and assign to it the gain $\alpha$. Then $E_0$ is the resulting set of oriented edges, and $\psi$ is the corresponding gain assignment.
(See \cite{jkt} for details.)

Suppose $\Gamma$ is an abstract multiplicative group. 
A closed walk $C=v_1,e_1,v_2,\dots,v_k,e_k,\break v_1$  in a quotient $\Gamma$-gain graph $(G_0,\psi)$ is called {\em balanced} if $\psi(C)=\Pi_{i=1}^k \psi(e_i)^{{\rm sign}(e_i)}=1$,
where ${\rm sign}(e_i)=1$ if $e_i$ is directed from $v_i$ to $v_{i+1}$, and ${\rm sign}(e_i)=-1$  otherwise. We say that an edge subset $F_0\subseteq E_0$ is {\em balanced} if all closed walks in $F_0$ are balanced; otherwise it is called {\em unbalanced}.


Let $k\in\mathbb{N}$, $l\in\{0,1,\ldots, 2k-1\}$ and $m\in \{0,1,\ldots, l\}$. Then  
 $(G_0,\psi)$ is called \emph{$(k,l,m)$-gain-sparse} if 
\begin{enumerate}
\item[(i)] $|F|\leq k|V(F)|-l$ for any nonempty balanced $F\subseteq E_0$, and,
\item[(ii)] $|F|\leq k|V(F)|-m$ for all $F\subseteq E_0$.
\end{enumerate}
Moreover, $(G_0,\psi)$ is \emph{$(k,l,m)$-gain-tight} if $|E(G_0)|=k|V(G_0)|-m$ 
and $(G_0,\psi)$ is $(k,l,m)$-gain-sparse.

\subsection{Schoenflies notation for symmetry groups on the 2-sphere} 

We call a subgroup of the orthogonal group $O(\mathbb{R}^d)$ a \emph{symmetry group} (in dimension $d$).
In the Schoenflies notation, the possible symmetry groups in dimension $3$ are $\mathcal{C}_s$, $\mathcal{C}_n$, $\mathcal{C}_i$, $\mathcal{C}_{nv}$,  $\mathcal{C}_{nh}$,  $\mathcal{D}_{n}$, $\mathcal{D}_{nh}$, $\mathcal{D}_{nd}$, $S_{2n}$, $\mathcal{T}$ , $\mathcal{T}_d$, $\mathcal{T}_h$,  $\mathcal{O}$, $\mathcal{O}_h$, $\mathcal{I}$ and $\mathcal{I}_h$.  $\mathcal{C}_s$ is generated by a single reflection $s$, and $\mathcal{C}_n$, $n\geq 1$, is a group
generated by an $n$-fold rotation $C_n$. $\mathcal{C}_i$ is the group generated by the inversion $\iota$, $\mathcal{C}_{nv}$ is a dihedral group that is generated by a rotation $C_n$ and
a reflection  whose reflectional plane contains the rotational axis of $C_n$, and 
 $\mathcal{C}_{nh}$ is generated by a rotation $C_n$ and the reflection  whose reflectional plane is perpendicular to the axis of $C_n$. Further, $\mathcal{D}_{n}$ denotes a symmetry group that
is generated by a rotation $C_n$ and another $2$-fold rotation $C_2$ whose rotational axis is perpendicular to the one of $C_n$. $\mathcal{D}_{nh}$ and $\mathcal{D}_{nd}$ are generated by the generators $C_n$ and $C_2$ of a group $\mathcal{D}_{n}$ and by a reflection $s$.  In the case of $\mathcal{D}_{nh}$, the
mirror of $s$ is the plane that is perpendicular to the $C_n$ axis and
contains the origin (and hence contains the rotational axis of $C_2$), whereas in
the case of $\mathcal{D}_{nd}$, the mirror of $s$ is a plane that contains the $C_n$ axis and forms an angle of $\frac{\pi}{n}$ with the $C_2$ axis.  $S_{2n}$ is a symmetry group which is generated by a $2n$-fold improper rotation (i.e., a rotation by $\frac{\pi}{n}$ followed by a reflection in the plane which is perpendicular to the rotational axis).
The remaining seven types of symmetry groups in dimension 3 are related
to the Platonic solids and are placed into three divisions: the tetrahedral
groups $\mathcal{T}$ , $\mathcal{T}_d$ and $\mathcal{T}_h$, the octahedral groups $\mathcal{O}$ and $\mathcal{O}_h$, and the icosahedral
groups $\mathcal{I}$ and $\mathcal{I}_h$. See \cite{Bis} for details.

The only possible symmetry groups in dimension $2$ are $\mathcal{C}_s$ (reflection symmetry), $\mathcal{C}_n$ (rotational symmetry) and $\mathcal{C}_{nv}$ (dihedral symmetry). In Section \ref{sec:higher} we will also consider certain types of symmetry groups in dimensions $4$ and higher, and we will also make use of the Schoenflies notation for these groups.

\subsection{Symmetric Euclidean frameworks}\label{subsec:symfw}

Let $\Gamma$ be an abstract group, and  let $G$ be a  $\Gamma$-symmetric graph with respect to
the action $\theta:\Gamma\rightarrow \textrm{ Aut}(G)$.
Suppose also that $\Gamma$ acts on $\mathbb{R}^d$ via a homomorphism $\tau:\Gamma\rightarrow O(\mathbb{R}^d)$.

A framework $(G,p)$ is called \emph{$\Gamma$-symmetric} (with respect to $\theta$ and $\tau$) if 
\begin{equation}
\tau(\gamma) (p(i))=p(\theta(\gamma) (i)) \qquad \textrm{for all } \gamma\in \Gamma \textrm{ and all } i\in V.
\end{equation}

A $\Gamma$-symmetric framework $(G,p)$ (with respect to $\theta$ and $\tau$) is called \emph{$\Gamma$-regular} if the rigidity matrix (i.e. the matrix corresponding to the linear system in (\ref{infinmotioneq})) has maximum rank among all realisations of $G$ as a $\Gamma$-symmetric framework (with respect to $\theta$ and $\tau$).

An infinitesimal motion $u$ of a $\Gamma$-symmetric framework $(G,p)$ is called \emph{$\Gamma$-symmetric} (with respect to $\theta$ and $\tau$) if the velocity vectors exhibit the same symmetry as $(G,p)$, that is, if $\tau(\gamma)u_i=u_{\gamma i}$ for all $\gamma\in \Gamma$ and all $i\in  V.$
We say that $(G,p)$ is \emph{forced $\Gamma$-symmetric infinitesimally rigid} if every $\Gamma$-symmetric infinitesimal motion is trivial.

An important motivation for studying forced $\Gamma$-symmetric infinitesimal rigidity is that for $\Gamma$-regular frameworks, there exists a non-trivial $\Gamma$-symmetric infinitesimal motion if and only if there exists a non-trivial symmetry-preserving \emph{continuous} motion \cite{BSfinite} (see also \cite{guestfow,kangguest}). A key tool to study forced $\Gamma$-symmetric infinitesimal rigidity is the so-called orbit matrix (see \cite{SWorbit} for details). With the help of this matrix, combinatorial characterisations for $\Gamma$-regular forced $\Gamma$-symmetric rigidity in the plane (where the action $\theta:\Gamma\to \textrm{Aut}(G)$ is free on the vertex set) have been obtained for the groups $\mathcal{C}_s$, $\mathcal{C}_n$, $n\in \mathbb{N}$, and $\mathcal{C}_{(2n+1)v}$, $n\in \mathbb{N}$, in \cite{jkt} (see also \cite{mt1}). In particular we have the following result for reflectional or rotational symmetry groups.

\begin{thm}\label{thm:forcedbjplane} Let $n\geq 2$ and let $(G,p)$ be a $\mathbb{Z}_n$-regular bar-joint framework in $\mathbb{R}^2$ with respect to the action $\theta:\mathbb{Z}_n\to \textrm{Aut}(G)$ (which acts freely on $V$) and $\tau:\mathbb{Z}_n\to O(\mathbb{R}^2)$. Then $(G,p)$ is forced $\mathbb{Z}_n$-symmetric infinitesimally rigid if and only if the quotient $\mathbb{Z}_n$-gain graph $(G_0,\psi)$ of $G$ contains a spanning subgraph that is $(2,3,1)$-gain-tight.
\end{thm}

 For the groups $\mathcal{C}_{(2n)v}$  the problem of finding a combinatorial characterisation for $\Gamma$-regular forced $\Gamma$-symmetric rigidity is still open \cite{jkt}.

If a $\Gamma$-symmetric framework is forced $\Gamma$-symmetric infinitesimally rigid, then it may still have non-trivial infinitesimal motions that are not $\Gamma$-symmetric. The problem of analysing the infinitesimal rigidity  of an (incidentally) $\Gamma$-symmetric framework can be broken up into independent subproblems, one for each irreducible representation of the group $\Gamma$, by an appropriate block-decomposition of the rigidity matrix. (The block matrix corresponding to the trivial representation of $\Gamma$ is the orbit matrix.) Combinatorial characterisations of $\Gamma$-regular infinitesimally rigid frameworks in the plane have been obtained via this approach for a selection of cyclic groups (where the action $\theta:\Gamma\to \textrm{Aut}(G)$ is free on the vertex set) \cite{Ikeshita,IkTan,ST}. The problem remains open for all other groups.

We offer a sample result for the groups $\mathcal{C}_s$ and $\mathcal{C}_2$, as we will discuss the relationship between these groups with respect to infinitesimal rigidity in greater detail in Sections~\ref{sec:pairs} and \ref{sec:non-free}.

\begin{thm}\label{thm:combincidentalbj} Let $n\geq 2$ and let $(G,p)$ be a $\mathbb{Z}_2$-regular bar-joint framework in $\mathbb{R}^2$ with respect to the action $\theta:\mathbb{Z}_2\to \textrm{Aut}(G)$ (which acts freely on $V$) and $\tau:\mathbb{Z}_2\to O(\mathbb{R}^2)$, where $\tau(\mathbb{Z}_2)=\mathcal{C}_s$ or $\mathcal{C}_2$. Then $(G,p)$ is infinitesimally rigid if and only if the quotient $\mathbb{Z}_2$-gain graph $(G_0,\psi)$ of $G$ contains a spanning  $(2,3,i)$-gain-tight subgraph $(H_i,\psi_i)$ for each $i=1,2$.
\end{thm}

\subsection{Symmetric frameworks on the sphere}\label{subsec:sph}

A spherical  framework $(G,p)$ in $\mathbb{S}^d$ is a bar-joint framework with  $p:V\rightarrow \mathbb{S}^d$, where the distance between two points is determined by their spherical distance, i.e. by their inner product (see Figure~\ref{fig:transfer_sph_ptline}). Alternatively, we may model $(G,p)$ as a `cone framework' $(G\star u,q)$ in $\mathbb{R}^{d+1}$. The cone graph $G\star u$ of $G$ is
 obtained from $G$ by adding the new cone vertex $u$ and the edges $\{u,v\}$
for all vertices $v\in V$. The cone framework $(G\star u,q)$ is obtained by fixing the cone vertex $u$ at the origin and setting $q|_V=p$. In the following we will assume that the points $p(V)$ linearly span $\mathbb{R}^{d+1}$. For the infinitesimal rigidity of such a framework we consider the linear system:
\begin{align}
\langle p_i, \dot{p}_j\rangle+\langle p_j, \dot{p}_i\rangle&=0 \qquad (\{i,j\}\in E) \label{eq:inner_inf}\\
\langle p_i, \dot{p}_i\rangle&=0 \qquad (i\in V). \label{eq:scale}
\end{align}
A map $\dot{p}:V\to \mathbb{R}^{d+1}$ is  said to be an {\em
infinitesimal motion} of $(G,p)$ if it satisfies this system of
linear constraints, and
$(G,p)$ is {\em infinitesimally
rigid} if the dimension of the space of its infinitesimal motions is
equal to ${d+1\choose 2}$ (i.e. every infinitesimal motion of $(G,p)$ is trivial).

A spherical framework $(G,p)$ in $\mathbb{S}^d$ is \emph{$\Gamma$-symmetric} (with respect to $\theta$ and $\tau$) if it is $\Gamma$-symmetric as a bar-joint framework in $\mathbb{R}^{d+1}$ (with respect to $\theta$ and $\tau$).  Forced $\Gamma$-symmetric infinitesimal rigidity for spherical frameworks is defined analogously as for bar-joint frameworks in $\mathbb{R}^d$.

A $\Gamma$-symmetric spherical framework $(G,p)$ (with respect to $\theta$ and $\tau$) in $\mathbb{S}^d$ is \emph{$\Gamma$-regular} if its spherical rigidity matrix (i.e. the matrix corresponding to the linear system  above) has maximum rank among all realisations of $G$ as a $\Gamma$-symmetric spherical framework (with respect to $\theta$ and $\tau$).

In \cite{NS}, combinatorial characterisations for $\Gamma$-regular forced $\Gamma$-symmetric rigidity on $\mathbb{S}^2$ (where the action $\theta:\Gamma\to \textrm{Aut}(G)$ is free on the vertex set)  have been established for the groups $\mathcal{C}_s$, $\mathcal{C}_n$, $n\in \mathbb{N}$, $\mathcal{C}_i$, $\mathcal{C}_{nv}$, $n$ odd, $\mathcal{C}_{nh}$, $n$ odd, and $\mathcal{S}_{2n}$, $n$ even. (For the groups $\mathcal{C}_s$ and $\mathcal{C}_n$, for example, the characterisation is the same as the one given in Theorem~\ref{thm:forcedbjplane} for bar-joint frameworks in $\mathbb{R}^2$.) For the remaining groups, this problem is still open. (See Table 1 in \cite{NS} for further details.) The infinitesimal rigidity for incidentally symmetric frameworks on $\mathbb{S}^2$ has not yet been investigated. We will discuss this further in Sections \ref{sec:transfer} and \ref{sec:pairs}. 

\subsection{Symmetric point-hyperplane frameworks}\label{subsec:ph}

Let  $G=(V_P\cup V_H, E)$ be a graph where the vertex set $V$ is partitioned into two sets $V_P$ and $V_H$.  This induces a partition of the edge set $E$  into the sets $E_{PP}, E_{PH}, E_{HH}$, where $E_{PP}$ consists of pairs of vertices in $V_P$, $E_{HH}$ consists of pairs of vertices in $V_H$, and $E_{PH}$ consists of pairs of vertices with one vertex in $V_P$ and the other one in $V_H$. We call such a graph $G$ a \emph{$PH$-graph}.

A \emph{point-hyperplane framework} in $\mathbb{R}^d$ is a triple $(G,p,\ell)$, where $G=(V_P\cup V_H, E)$ is a $PH$-graph, and $p:V_P\to \mathbb{R}^d$ and $\ell=(a,r):V_H\to \mathbb{S}^{d-1}\times \mathbb{R}$ are maps. These maps $p$ and $\ell$ are interpreted as follows: each vertex $i$ in $V_P$ is mapped to the point $p_i$ in $\mathbb{R}^d$ and each vertex $j$ in $V_H$ is mapped to the hyperplane in $\mathbb{R}^d$ given by  $\{x\in \mathbb{R}^d: \langle a_j, x\rangle+r_j=0\}$. A point-hyperplane framework in $\mathbb{R}^2$ is also called a \emph{point-line framework} \cite{JO} (see Figure~\ref{fig:transfer_sph_ptline}). In the following we will assume that the points $p(V_P)$ and hyperplanes $\ell(V_H)$ affinely span $\mathbb{R}^d$. Each edge in $E_{PP},E_{PH},E_{HH}$ indicates a point-point distance constraint, a point-hyperplane distance constraint, or a hyperplane-hyperplane
angle constraint, respectively. This leads to the following system of first order constraints (see \cite{EJNSTW} for details):
\begin{align}
\langle p_i - p_j, \dot{p}_i-\dot{p}_j\rangle&=0 && (\{i,j\}\in E_{PP}) \label{eq:line_inf1_euc} \\
\langle p_i, \dot{a}_j\rangle+\langle \dot{p}_i, a_j\rangle +\dot{r}_j&=0 && (\{i,j\}\in E_{PH}) \label{eq:line_inf2_euc}\\
\langle a_i, \dot{a}_j\rangle+\langle \dot{a}_i, a_j\rangle&=0 && (\{i,j\}\in E_{HH}) \label{eq:line_inf3_euc}\\
\langle a_i, \dot{a}_i\rangle&=0 && (i\in V_H).
\label{eq:a_inf_euc}
\end{align}
A map $(\dot{p},\dot\ell)$ is  said to be an {\em
infinitesimal motion} of $(G,p,\ell)$ if it satisfies this system of
linear constraints, and
$(G,p,\ell)$ is {\em infinitesimally
rigid} if the dimension of the space of its infinitesimal motions is
equal to ${d+1\choose 2}$ (i.e. every infinitesimal motion of $(G,p,\ell)$ is trivial).

\begin{rem}\label{rem:shift} As discussed in \cite{EJNSTW}, translating a hyperplane in a  point-hyperplane framework does not affect its infinitesimal rigidity properties. We may therefore assume without loss of generality that every hyperplane contains the origin.
\end{rem}

Let $G=(V_P\cup V_H, E)$ be a $PH$-graph. A \emph{$PH$-stabilising automorphism} of $G$ is an automorphism $\pi\in \textrm{Aut}(G)$ such that $\pi(x)\in V_P$ for all $x\in V_P$ and $\pi(y)\in V_H$ for all $v\in V_H$. The subgroup of all $\pi\in\textrm{Aut}(G)$ that are $PH$-stabilising is denoted by $\textrm{Aut}_{PH}(G)$. 
We only consider a $PH$-graph $G$ to be $\Gamma$-symmetric if there exists a group action $\theta:\Gamma\to \textrm{Aut}_{PH}(G)$.

Let $G=(V_P\cup V_H, E)$ be a $\Gamma$-symmetric $PH$-graph with respect to $\theta:\Gamma\to \textrm{Aut}_{PH}(G)$. Further, let $(G,p,\ell)$ be a point-hyperplane framework in $\mathbb{R}^d$ and suppose $\Gamma$ acts on $\mathbb{R}^d$ via a homomorphism $\tau:\Gamma\to O(\mathbb{R}^d)$. Then $(G,p,\ell)$ is called \emph{$\Gamma$-symmetric} (with respect to $\theta$ and $\tau$) if 
\begin{align}
\tau(\gamma) (p(i))&=p(\theta(\gamma) (i)) && \textrm{for all } \gamma\in \Gamma \textrm{ and all } i\in V_P\label{phsym1}\\
\tau(\gamma) (a(j))&=\pm a(\theta(\gamma) (j)) && \textrm{for all } \gamma\in \Gamma \textrm{ and all } j\in V_H\label{phsym2}\\
r(j)&=r(\theta(\gamma) (j)) && \textrm{for all } \gamma\in \Gamma \textrm{ and all } j\in V_H.
\label{phsym3}
\end{align}
An infinitesimal motion $(\dot{p},\dot\ell)$ of a $\Gamma$-symmetric point-hyperplane framework $(G,p,\ell)$ is called $\Gamma$-symmetric if it satisfies the constraints in (\ref{phsym1})-(\ref{phsym3}) and $(G,p,\ell)$  is called \emph{forced $\Gamma$-symmetric infinitesimally rigid} if every $\Gamma$-symmetric infinitesimal motion is trivial.
A $\Gamma$-symmetric point-hyperplane framework $(G,p,\ell)$ (with respect to $\theta$ and $\tau$) is \emph{$\Gamma$-regular} if its point-hyperplane rigidity matrix (i.e. the matrix corresponding to the linear system  (\ref{eq:line_inf1_euc})-(\ref{eq:a_inf_euc}) has maximum rank among all  realisations of $G$ as a $\Gamma$-symmetric point-hyperplane framework (with respect to $\theta$ and $\tau$).

The infinitesimal rigidity for incidentally or forced $\Gamma$-symmetric point-hyperplane frameworks has not yet been investigated. We will address these questions in the remaining sections of this paper. In particular, we will establish combinatorial characterisations for incidental and forced $\Gamma$-symmetric infinitesimal rigidity for some special classes of point-line frameworks in Sections~\ref{sec:pairs} and \ref{sec:non-free}.

\begin{figure}[htp]
        \begin{center} \includegraphics[scale=0.3]{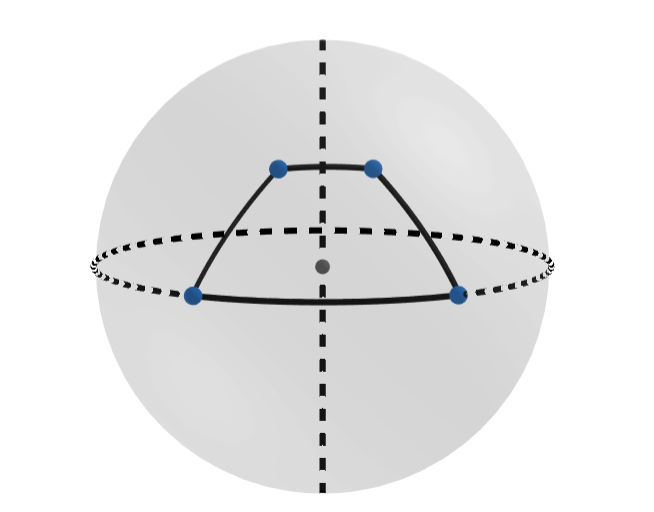} \hspace{0.2cm} 
         \includegraphics[scale=0.39]{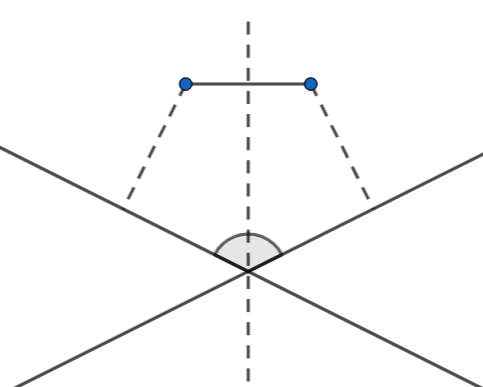}\hspace{0.2cm}
         \includegraphics[scale=0.45]{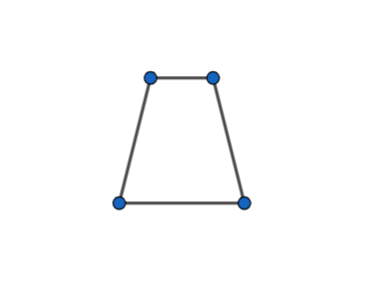} \hspace{0.1cm} 
         \includegraphics[scale=0.45]{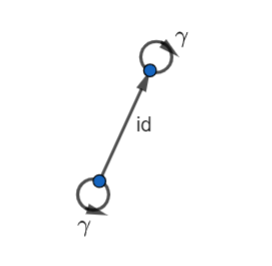}
        \end{center}   
\caption{A  spherical  framework $(G,q)$ with $\mathcal{C}_s$ symmetry in $\mathbb{S}^2$ with two points on the equator and the corresponding point-line framework $(G,p,\ell)$ with $\mathcal{C}_s$ symmetry in the affine plane $\mathbb{A}^2$ obtained from central projection. Both frameworks are infinitesimally flexible, but forced $\mathbb{Z}_2$-symmetric infinitesimally rigid. The underlying graph $G$ and its quotient $\mathbb{Z}_2$-gain graph (with $\mathbb{Z}_2=\langle \gamma \rangle$) are shown on the right.}
\label{fig:transfer_sph_ptline}
\end{figure}




\section{Transfer of infinitesimal rigidity}
\label{sec:transfer}

We first state a basic lemma which will be used repeatedly throughout this paper. 

\begin{lem}[\cite{BSWWSphere}]\label{lem:inv} 
Let $G=(V,E)$ be a graph and let $I\subseteq V$.
For a vector $q\in \mathbb{R}^{d+1}$, let $\iota$ denote the inversion operator defined by taking $(\iota \circ q)_i=-q_i$ if $i\in I$ and $(\iota \circ q)_i=q_i$ otherwise. If $(G,p)$ and $(G,\iota \circ p)$ are two frameworks on $\mathbb{S}^d$ then $(G,p)$ is infinitesimally rigid if and only if $(G,\iota \circ p)$ is infinitesimally rigid.
\end{lem}

Note that the proof uses the fact that the framework is on the sphere in an essential way.  The proof also shows that all other rigidity properties are preserved, including independence of rows, and isomorphic spaces of infinitesimal motions,  We will return to this in Section~7.4. (See Section~3.3 in \cite{BSWWSphere} for details.)

Our first theorem extends the transfer of Theorem \ref{thm:regulartransfer} (a),(b) to symmetric frameworks. (Note that the special case where there are no points on the equator (i.e. points that centrally project to points at infinity) was proved in  \cite{BSWWSphere}). We need the following definitions. For the sphere $\mathbb{S}^d$, we call the intersection of $\mathbb{S}^d$ with the linear hyperplane of $\mathbb{R}^{d+1}$ with normal vector $\be=(0,\ldots, 0,1)$ the \emph{equator} of $\mathbb{S}^d$.
Moreover, for a group $\Gamma$ and a representation $\tau:\Gamma\to O(\mathbb{R}^d)$, we let $\tilde\tau:\Gamma\rightarrow O(\mathbb{R}^{d+1})$ be the \emph{augmented representation} of $\tau$, i.e.,  $\tilde\tau(\gamma)=\begin{pmatrix} \tau(\gamma) & 0 \\ 0 & 1 \end{pmatrix}$.

\begin{thm}\label{thm:symtrans}
Let $G=(V,E)$ be a graph and $X\subseteq V$. Further, let $\tau(\Gamma)$ be a symmetry group in $\mathbb{R}^d$. 
 Then the following are equivalent:
\begin{itemize}
\item[(a)] $G$ can be realised as an infinitesimally rigid $\Gamma$-symmetric bar-joint framework on $\mathbb{S}^d$ (with respect to $\theta$ and $\tilde \tau$) such that the points assigned to $X$ lie on the equator.
\item[(b)] $G$ can be realised as an infinitesimally rigid $\Gamma$-symmetric point-hyperplane framework in $\mathbb{R}^d$ (with respect to $\theta$ and $\tau$) such that each vertex in $X$ is realised as a hyperplane and each vertex in $V\setminus X$ is realised as a point.
\end{itemize}
\end{thm}

\begin{proof}
Given a point-hyperplane framework $(G,p,\ell)$ in $\mathbb{R}^d$, we may construct a corresponding spherical framework $(G,q)$ with all points in the upper hemisphere
by setting $q(i)=\frac{\hat p_i}{\|\hat p_i\|}$, where $\hat p_i=(p_i,1)$, for all $i\in V_P$, and $q(j)=(a_j,0)$ for all $j\in V_H$.
It was shown in \cite{EJNSTW,BSWWSphere} that $(G,p,\ell)$ is infinitesimally rigid in $\mathbb{R}^d$ if and only if $(G,q)$ is infinitesimally rigid in $\mathbb{S}^d$
with all points in the upper hemisphere.  
We show that this operation also preserves the $\Gamma$ symmetry. 

Suppose $(G,p,\ell)$ is $\Gamma$-symmetric with respect to $\theta$ and $\tau$, i.e. equations (\ref{phsym1})-(\ref{phsym3}) are satisfied. Without loss of generality, we may assume that the normal vectors of the hyperplanes,  $a_j$, $j\in V_H$,  are oriented in such a way that we have a plus sign on the right hand side of equation (\ref{phsym2}).

 Let $i\in V_P$. Then for all $\gamma\in \Gamma$ we have $\|\hat p(i)\|=\|\tilde\tau(\gamma)\hat p(i)\|$ and $\tilde\tau(\gamma)\hat p(i)=\hat p(\theta(\gamma)(i))$. Thus,
$$\tilde\tau(\gamma) (q(i))= \tilde \tau(\gamma) \big(\frac{\hat p(i)}{\|\hat p(i)\|}\big)=\frac{1}{\|\hat p(i)\|}  \tilde\tau(\gamma) (\hat p(i)) =  \frac{1}{\|\hat p(\theta(\gamma)(i))\|}  \hat p(\theta(\gamma)( i)) = q(\theta(\gamma) (i)).$$

Now let $j\in V_H$. Then for all $\gamma\in \Gamma$ we have
$$\tilde \tau(\gamma)(q(j))=\tilde\tau(\gamma)((a(j),0))=(\tau(\gamma)(a(j)),0)=(a(\theta(\gamma)(j)),0)=q(\theta(\gamma)(j)).$$ 
This says that $(G,q)$ is $\Gamma$-symmetric with respect to $\theta$ and $\tilde\tau$, as desired.
 
 Conversely, if $(G,q)$ is $\Gamma$-symmetric with respect to $\theta$ and $\tilde\tau$, then it follows from $\tilde\tau(\gamma)(q(i))=q(\theta(\gamma)(i))$ for $i\in V\setminus X$ that $\tau(\gamma)(p(i))=p(\theta(\gamma)(i))$ for all $\gamma\in \Gamma$. Similarly, it follows from $\tilde\tau(\gamma)(q(j))=q(\theta(\gamma)(j))$ for $j\in X$ that $\tau(\gamma)(a(j))=a(\theta(\gamma)(j))$ for all $\gamma \in \Gamma$. Moreover, we set $r(j)=r(\theta(\gamma)( j))$ for all  $\gamma\in \Gamma$. Then   $(G,p,\ell)$ with $V_H=X$ and $V_P=V\setminus X$ is $\Gamma$-symmetric with respect to $\theta$ and $\tau$.

Finally, if we start with a $\Gamma$-symmetric spherical framework (with respect to $\theta$ and $\tilde \tau$) that has points above and below the equator, then, by definition of $\tilde{\tau}$, the vertices in a vertex orbit lie either all above, or all below, or all on the equator. Therefore, we may use Lemma~\ref{lem:inv} to invert all vertex orbits in the strict lower hemisphere to the upper hemisphere, preserving the symmetry and infinitesimal rigidity.
 
\end{proof}

Let $\Gamma$ be a group,  $\tau:\Gamma\to O(\mathbb{R}^d)$ be a representation, and $\tilde{\tau}$ be the augmented representation. For a $\Gamma$-symmetric graph $G=(V,E)$ (with respect to $\theta$) and a (possibly empty) set $X\subseteq V$, we say that a $\Gamma$-symmetric spherical framework (with respect to $\theta$ and $\tilde{\tau}$) with all points assigned to $X$ lying on the equator of $\mathbb{S}^d$ is \emph{$\Gamma$-$X$-regular} if the spherical rigidity matrix has maximum rank among all realisations of $G$ as a $\Gamma$-symmetric spherical framework (with respect to $\theta$ and $\tilde{\tau}$) with points assigned to $X$ lying on the equator. Clearly, a $\Gamma$-regular spherical framework is also $\Gamma$-$X$-regular. The converse, however, is in general not true.

Using techniques similar to \cite{BSWWSphere} we can see that the transfer above takes $\Gamma$-$X$-regular spherical frameworks to $\Gamma$-regular point-hyperplane frameworks.

\begin{lem}\label{lem:regtran}
Let $(G,q)$ be a $\Gamma$-symmetric framework on $\mathbb{S}^d$, with points assigned to a (possibly empty) subset $X$ of $V$ lying on the equator, and let $(G,p,\ell)$ be the corresponding $\Gamma$-symmetric point-hyperplane framework in $\mathbb{R}^d$ resulting from the transfer in Theorem \ref{thm:symtrans}. Then $(G,p,\ell)$ is $\Gamma$-regular if and only if $(G,q)$ is $\Gamma$-$X$-regular.
\end{lem}
\begin{proof} By Theorem \ref{thm:symtrans}, $q$ gives the maximum rank for the spherical rigidity matrix for $G$ (among all  $\Gamma$-symmetric realisations of $G$ on $\mathbb{S}^d$  with points assigned to $X$ lying on the equator) if and only if $(p,\ell)$ gives the maximum rank of the point-hyperplane rigidity matrix for $G$ (among all $\Gamma$-symmetric point-hyperplane realisations of $G$ in $\mathbb{R}^d$ with $V_P=V\setminus X$ and $V_H=X$). Moreover, moving in open neighborhoods of $q$ within the space of $\Gamma$-symmetric realisations of $G$ in $\mathbb{S}^d$ with points assigned to $X$ lying on the equator, and of $(p,\ell)$ within the space of $\Gamma$-symmetric point-hyperplane realisations of $G$ in $\mathbb{R}^d$, respectively, the rank of the rigidity matrices cannot drop immediately, but must be maintained over an open set.
\end{proof}

From Theorem~\ref{thm:symtrans} and Lemma~\ref{lem:regtran}, we immediately obtain the following corollary.

\begin{cor}\label{cor:combincidental} Let  $G=(V,E)$ be a graph and $X\subseteq V$. Further, let $\tau(\Gamma)$ be a symmetry group in $\mathbb{R}^d$.
\begin{itemize}
\item[(a)] If $X\neq \emptyset$, then $\Gamma$-$X$-regular realisations of $G$ as a spherical framework on $\mathbb{S}^d$ (with respect to $\theta$ and $\tilde \tau$) are infinitesimally rigid if and only if $\Gamma$-regular realisations of $G$ as a point-hyperplane framework in $\mathbb{R}^d$ (with respect to $\theta$ and $\tau$) with $V_P=V\setminus X $ and $V_H=X$ are infinitesimally rigid.
\item[(b)] If $X= \emptyset$, then $\Gamma$-regular realisations of $G$ as a spherical framework on $\mathbb{S}^d$ (with respect to $\theta$ and $\tilde \tau$) are infinitesimally rigid if and only if $\Gamma$-regular realisations of $G$ as a bar-joint framework in $\mathbb{R}^d$ (with respect to $\theta$ and $\tau$) are infinitesimally rigid.
\end{itemize}
\end{cor}

\begin{rem}\label{rem1}
Since we have combinatorial characterisations of $\Gamma$-regular infinitesimally rigid frameworks in $\mathbb{R}^2$ (where the action $\theta:\Gamma\to \textrm{Aut}(G)$ is free on the vertex set)  for the groups $\mathcal{C}_s$, $\mathcal{C}_2$ and $\mathcal{C}_n$, $n$ odd \cite{ST}  (recall also Theorem~\ref{thm:combincidentalbj}), those results, together with Corollary~\ref{cor:combincidental} (b), immediately provide us with the corresponding combinatorial characterisations of $\Gamma$-regular infinitesimally rigid spherical frameworks on $\mathbb{S}^2$ for these groups.

However, we can only deduce complete combinatorial results regarding the infinitesimal rigidity of point-line frameworks from Corollary~\ref{cor:combincidental} (a) in some very special cases (see Theorems~\ref{thm:lamanptline1} and \ref{thm:lamanptline2}). This is because a $\Gamma$-$X$-regular spherical framework is in general not a $\Gamma$-regular spherical framework (even when $|X|=2$), and hence the combinatorial results for $\Gamma$-regular bar-joint frameworks in $\mathbb{R}^2$ (such as the ones mentioned above) do not apply here. Consider, for example, a framework $(G,p)$ on $\mathbb{S}^2$ with $\mathcal{C}_2$ symmetry, where the half-turn swaps two points $p_i$ and $p_j$, with $\{i,j\}\in E$. If $p_i$ and $p_j$ lie on the equator, then this edge will always be redundant, whereas otherwise this is not the case. 


\end{rem}

For the reflection group $\mathcal{C}_s$ in $\mathbb{R}^d$, we also obtain the following complete analogue of Theorem~\ref{thm:regulartransfer}. 

\begin{cor}\label{cor:mirrortransfer}
Let $G=(V,E)$ be a graph,  $X\subseteq V$, and $\Gamma=\mathbb{Z}_2$. Further let $\tau(\mathbb{Z}_2)$ be the symmetry group $\mathcal{C}_s$ in $\mathbb{R}^d$. 
Then the following are equivalent:
\begin{itemize}
\item[(a)] $G$ can be realised as an infinitesimally rigid $\Gamma$-symmetric bar-joint framework on $\mathbb{S}^d$ (with respect to $\theta$ and $\tilde\tau$) such that the points assigned to $X$ lie on the equator, but not on the line through the origin that is perpendicular to the mirror hyperplane.
\item[(b)] $G$ can be realised as an infinitesimally rigid $\Gamma$-symmetric point-hyperplane framework in $\mathbb{R}^d$ (with respect to $\theta$ and $\tau$) such that each vertex in $X$ is realised as a hyperplane,  no hyperplane is parallel to the mirror hyperplane, and each vertex in $V\setminus X$ is realised as a point.
\item[(c)]  $G$ can be realised as an infinitesimally rigid $\Gamma$-symmetric bar-joint framework in $\mathbb{R}^d$ (with respect to $\theta$ and $\tau$) such that the points assigned to $X$ lie on a hyperplane (perpendicular to the mirror hyperplane).
\end{itemize}
\end{cor}

\begin{proof} The equivalence of (a) and (b) follows immediately from Theorem \ref{thm:symtrans}. It remains to show that (b) and (c) are equivalent.

It was shown in \cite{EJNSTW} that $(G, p,\ell)$ is infinitesimally rigid as a point-hyperplane framework in $\mathbb{R}^d$ if and only if $(G,q^{-1}\circ \iota \circ \gamma \circ q)$ is infinitesimally rigid as a bar-joint framework in $\mathbb{R}^d$, where $q$ is obtained from $(p,\ell)$ as described in the proof of Theorem \ref{thm:symtrans} (and $q^{-1}$ denotes the inverse function), $\gamma$ is a rotation in $\mathbb{R}^{d+1}$ about an axis through the origin, and $\iota$ is the inversion operator defined by taking $(\iota \circ q)_i=-q_i$ if $i\in I$ and $(\iota \circ q)_i=q_i$ otherwise. It remains to show that these operations can be performed while preserving the mirror symmetry.  To preserve the mirror symmetry, the rotation $\gamma$ must be around the axis that is perpendicular to the mirror hyperplane.  Since, by assumption, there is no vertex on that axis, all points can be moved off the equator by rotating around that axis. We can clearly now use the inversion operator $\iota$ to move all points onto the strict upper hemisphere while preserving the mirror symmetry. This gives the result.
\end{proof}

Note that for any group containing a rotation the operation $\gamma$ will destroy the symmetry so the proof of Corollary~\ref{cor:mirrortransfer} is not sufficient to handle other groups. 


\section{Transfer of forced-symmetric infinitesimal rigidity}
\label{sec:forced}

As is standard for discussions on forced-symmetric rigidity,  we will assume for simplicity  throughout this section that $G=(V,E)$ is a $\Gamma$-symmetric graph with respect to $\theta$, where $\theta$ acts freely on $V$.

First we state the forced-symmetric analogue of Lemma~\ref{lem:inv}, which was also proved in \cite{BSWWSphere}. 

\begin{lem}\label{lem:invforced} Let $G=(V,E)$ be a $\Gamma$-symmetric graph with respect to $\theta:\Gamma\to \textrm{Aut}(G)$, and let $\tau(\Gamma)$ be a symmetry group in dimension $d$. Further let $I\subseteq V$ be a set of vertex orbits under the group action $\theta$.
For a vector $q\in \mathbb{R}^{d+1}$, let $\iota$ denote the inversion operator defined by taking $(\iota \circ q)_i=-q_i$ if $i\in I$ and $(\iota \circ q)_i=q_i$ otherwise. 
Then $(G,p)$ is a $\Gamma$-symmetric framework on $\mathbb{S}^d$ (with respect to $\theta$ and $\tilde{\tau}$) if and only if $(G,\iota \circ p)$ is. Moreover,  $(G,p)$ is forced $\Gamma$-symmetric infinitesimally rigid if and only if $(G,\iota \circ p)$ is forced $\Gamma$-symmetric infinitesimally rigid.
\end{lem}

Next we will extend the transfer results of Section~\ref{sec:transfer} to the context of forced $\Gamma$-symmetric rigidity, where the action is free on the vertices, by adapting the approach in \cite{EJNSTW}.

Let $(G,p)$ be a $\Gamma$-symmetric spherical framework and let $(G_0,\psi)=(V_0,E_0,\psi)$ be the $\Gamma$-gain graph of $G$. In the following we identify $V_0$ with a set of representative vertices for the vertex orbits under $\Gamma$.  Recall that $(G_0,\psi)$ is a directed (group-labeled) multigraph, so we denote an edge from a vertex  $i$ to a (not necessarily distinct) vertex $j$ by $(i,j)$. By definition, a $\Gamma$-symmetric infinitesimal motion $\dot{p}$ of $(G,p)$ satisfies the following linear system:

\begin{align}
\label{eq:inner_inf1}
\langle p_i, \tau(\psi((i,j)))\dot{p}_j\rangle+\langle \tau(\psi((i,j)))p_j, \dot{p}_i\rangle&=0 \qquad ((i,j)\in E_0) \\
\langle p_i, \dot{p}_i\rangle&=0 \qquad (i\in V_0). \label{eq:scale1}
\end{align}

In the following we will simplify notation by setting $\psi((i,j))=\psi_{ij}$.
For a $\Gamma$-symmetric point-hyperplane framework $(G,p,\ell)$ in $\mathbb{R}^d$, we first show the different types of geometric constraints to help the reader see where the linear system for a $\Gamma$-symmetric infinitesimal motion comes from:
\begin{align}
\label{eq:line_const1_euc}
\| p_i-\tau(\psi_{ij})p_j\|^2&=\text{const} && ((i,j)\in E_{0_{PP}}) \\
|\langle p_i, \tau(\psi_{ij})a_j\rangle+r_j|&=\text{const} && ((i,j)\in E_{0_{PH}}) \label{eq:line_const2_euc}\\
\langle a_i, \tau(\psi_{ij})a_j\rangle&=\text{const} && ((i,j)\in E_{0_{HH}}). \label{eq:line_const3_euc}
\end{align}
Since  $a_i\in\mathbb{S}^{d-1}$,
we also have the constraint
\begin{align*}
\langle a_i, {a}_i\rangle&=1 & (i\in V_{0_H}).
\end{align*}
Taking derivatives we get the following system of first order constraints (recall also Section~\ref{subsec:ph}):
\begin{align}
\langle p_i - \tau(\psi_{ij})p_j, \dot{p}_i-\tau(\psi_{ij})\dot{p}_j\rangle&=0 &&  ((i,j)\in E_{0_{PP}}) \label{eq:line_inf1_euc} \\
\langle p_i, \tau(\psi_{ij})\dot{a}_j\rangle+\langle \dot{p}_i, \tau(\psi_{ij})a_j\rangle +\dot{r}_j&=0 && ((i,j)\in E_{0_{PH}}) \label{eq:line_inf2_euc}\\
\langle a_i, \tau(\psi_{ij})\dot{a}_j\rangle+\langle \dot{a}_i, \tau(\psi_{ij})a_j\rangle&=0 &&  ((i,j)\in E_{0_{HH}}) \label{eq:line_inf3_euc}\\
\langle a_i, \dot{a}_i\rangle&=0 && (i\in V_{0_H}).
\label{eq:a_inf_euc}
\end{align}

We now translate $(G,p,\ell)$ to the point-hyperplane framework $(G,\hat{p},\ell)$ in affine space $\mathbb{A}^d$ by setting $\hat p_i=(p_i,1)$ for all $i\in V_{0_P}$. The system of constraints (\ref{eq:line_inf1_euc})-(\ref{eq:a_inf_euc}) then becomes:
\begin{align}
\langle \hat p_i - \tilde{\tau}(\psi_{ij})\hat p_j, \dot{\hat p}_i-\tilde{\tau}(\psi_{ij})\dot{\hat p}_j\rangle&=0 &&  ((i,j)\in E_{0_{PP}}) \label{eq:line_inf1_euc1} \\
\langle \hat p_i, \tilde{\tau}(\psi_{ij})\dot{\ell}_j\rangle+\langle \dot{\hat p}_i, \tilde{\tau}(\psi_{ij})\ell_j\rangle &=0 && ((i,j)\in E_{0_{PH}}) \label{eq:line_inf2_euc1}\\
\langle a_i, \tau(\psi_{ij})\dot{a}_j\rangle+\langle \dot{a}_i, \tau(\psi_{ij})a_j\rangle&=0 &&  ((i,j)\in E_{0_{HH}}) \label{eq:line_inf3_euc1}\\
\langle \dot{\hat p}_i, \bf{e}\rangle&=0 && (i\in V_{0_P})
\label{eq:a_inf_euc2}\\
\langle a_i, \dot{a}_i\rangle&=0 && (i\in V_{0_H}).
\label{eq:a_inf_euc1}
\end{align}
where $\bf{e}$ is the vector whose last coordinate is 1 and all others are equal to 0.

As in \cite{EJNSTW} the last coordinate
of $\ell_i$ is not important when analyzing the infinitesimal
rigidity of $(G,p,\ell)$ (recall also Remark~\ref{rem:shift}), and we may always assume that $\ell$ is a
map with $\ell:V_L\rightarrow \mathbb{S}^{d-1}\times \{0\}$. Under
this assumption, we can regard each $\ell_i$ as a point on the
equator $Q$ of $\mathbb{S}^{d}$ by identifying
$\mathbb{S}^{d-1}\times \{0\}$ with $Q$.
Hence (\ref{eq:a_inf_euc1}) can be written as $\langle \ell_i, \dot{\ell}_i\rangle=0$, i.e.
$\dot{\ell}_i\in T_{\ell_i}\mathbb{S}^d$ for all $i\in V_{0_H}$, where $T_xY$ (or simply $TY$ if $x$ is not relevant) denotes the tangent hyperplane at the point $x$ to the space $Y$. Moreover,
(\ref{eq:line_inf3_euc1}) gives
\[
\langle \ell_i,  \tilde{\tau}(\psi_{ij})\dot{\ell}_j\rangle+\langle \dot{\ell}_i,  \tilde{\tau}(\psi_{ij})\ell_j\rangle=0
\]
for all $(i,j)\in E_{0_{HH}}$.

Let $\mathbb{S}^d_{>0}$ denote the strict upper hemisphere of $\mathbb{S}^d$ and define
$\phi:\mathbb{A}^d\rightarrow \mathbb{S}^d_{>0}$ to be the central
projection, that is,
\begin{equation*}
\label{eq:phi1}
\phi(x)=\frac{x}{\|x\|}\qquad (x\in \mathbb{A}^d),
\end{equation*}
and for each $x\in \mathbb{A}^d$, define $\chi_x:T\mathbb{A}^d\rightarrow T_{\phi(x)}\mathbb{S}^d$ by
\[
\chi_x(m)=\frac{m-\langle m, x\rangle {\bf e}}{\|x\|} \qquad (m\in T\mathbb{A}^d).
\]

 It was shown in \cite{BSWWSphere} that Equation (\ref{eq:line_inf1_euc1}) can be rewritten as
\[
\langle \phi(\hat{p}_i),
\chi_{\tilde{\tau}(\psi_{ij})\hat{p}_j}(\tilde{\tau}(\psi_{ij})\dot{\hat{p}}_j)\rangle+\langle \phi((\tilde{\tau}(\psi_{ij})\hat{p}_j),
\chi_{\hat{p}_i}(\dot{\hat{p}}_i)\rangle=\frac{\langle \hat{p}_i-\tilde{\tau}(\psi_{ij})\hat{p}_j,
\dot{\hat{p}}_i-\tilde{\tau}(\psi_{ij})\dot{\hat{p}}_j\rangle}{\|\hat{p}_i\|\|\hat{p}_j\|}= 0
\]
for all $(i,j)\in E_0$ with $i,j\in V_{0_P}$. As in \cite{EJNSTW}
Equation (\ref{eq:line_inf2_euc1}) can also be rewritten as
\[
\langle \phi(\hat{p}_i), \tilde{\tau}(\psi_{ij})\dot{\ell}_j\rangle+\langle \psi_{\hat{p}_i}(\dot{\hat{p}}_i), \tilde{\tau}(\psi_{ij})\ell_j \rangle=
\frac{\langle \hat{p}_i, \tilde{\tau}(\psi_{ij})\dot{\ell}_j\rangle+\langle \dot{\hat{p}}_i, \tilde{\tau}(\psi_{ij})\ell_j\rangle}{\|\hat{p}_i\|}=0
\]
for all $(i,j)\in E_0$ with $i\in V_{0_P}$ and $j\in V_{0_H}$.

These equations imply that $(\dot{\hat{p}}, \dot{\ell})$ is a $\Gamma$-symmetric infinitesimal motion of $(G,\hat{p},\ell)$ if and only if
$\dot{q}$ is  a $\Gamma$-symmetric infinitesimal motion of $(G,q)$, where
$(G,q)$ is the bar-joint framework on $\mathbb{S}_{\geq 0}^d$ (i.e., the upper hemisphere including the equator) given by
\begin{equation}
\label{eq:phi2}
q_i=\begin{cases}
\phi(\hat{p}_i) & (i\in V_{0_P}) \\
(a_i,0) & (i\in V_{0_H}),
\end{cases}
\end{equation}
and
$\dot{q}_i\in T_{q_i}\mathbb{S}^d$ is given by
\begin{equation*}
\label{eq:phi3}
\dot{q}_i=\begin{cases}
\chi_{\hat{p}_i}(\dot{\hat{p}}_i) & (i\in V_{0_P}) \\
\dot{\ell}_i & (i\in V_{0_H}).
\end{cases}
\end{equation*}

 Since each $\chi_x$ is bijective and hence invertible, this gives us an isomorphism between the spaces of infinitesimal motions of
$(G,\hat{p},\ell)$ and $(G, q)$.  Moreover, by applying the above isomorphism to a framework on the complete graph that affinely spans $\mathbb{A}^d$, we see that the spaces of trivial $\Gamma$-symmetric infinitesimal motions have the same dimension. Finally,  we can simply identify $\mathbb{A}^d$ with $\mathbb{R}^d$, i.e. the infinitesimal rigidity properties of $(G,\hat{p},\ell)$ in $\mathbb{A}^d$ are the same as for $(G,p,\ell)$ in $\mathbb{R}^d$ .

As in \cite[Theorem 2.2]{EJNSTW} the above discussion allows us to obtain the following analogue of Theorem~\ref{thm:symtrans}.

\begin{thm}
Let $G=(V,E)$ be a $\Gamma$-symmetric graph (with respect to $\theta$), where the action $\theta$ acts freely on $V$. Further, let $X\subseteq V$, and let $\tau(\Gamma)$ be a symmetry group in $\mathbb{R}^d$.  Then the following are equivalent:
\begin{itemize}
\item[(a)] $G$ can be realised as a forced $\Gamma$-symmetric infinitesimally rigid bar-joint framework on $\mathbb{S}^d$ (with respect to $\theta$ and $\tilde \tau$) such that the points assigned to $X$ lie on the equator.
\item[(b)]  $G$ can be realised as a  forced $\Gamma$-symmetric infinitesimally rigid  point-hyperplane framework in $\mathbb{R}^d$ (with respect to $\theta$ and $\tau$) such that each vertex in $X$ is realised as a hyperplane and each vertex in $V\setminus X$ is realised as a point.
\end{itemize}
\end{thm}

 As in Corollary \ref{cor:combincidental} we may deduce the following.

\begin{cor}\label{cor:combforced} 
Let  $G=(V,E)$ be a $\Gamma$-symmetric graph (with respect to $\theta$), where the action $\theta$ acts freely on $V$. Further, let $X\subseteq V$, and let $\tau(\Gamma)$ be a symmetry group in $\mathbb{R}^d$.
\begin{itemize}
\item[(a)] If $X\neq \emptyset$, then $\Gamma$-$X$-regular realisations of $G$ as a spherical framework on $\mathbb{S}^d$ (with respect to $\theta$ and $\tilde \tau$) are forced $\Gamma$-symmetric infinitesimally rigid if and only if $\Gamma$-regular realisations of $G$ as a point-hyperplane framework in $\mathbb{R}^d$ (with respect to $\theta$ and $\tau$) with $V_P=V\setminus X $ and $V_H=X$ are forced $\Gamma$-symmetric infinitesimally rigid.
\item[(b)] If $X= \emptyset$, then $\Gamma$-regular realisations of $G$ as a spherical framework on $\mathbb{S}^d$(with respect to $\theta$ and $\tilde \tau$) are forced $\Gamma$-symmetric infinitesimally rigid if and only if $\Gamma$-regular realisations of $G$ as a bar-joint framework in $\mathbb{R}^d$ (with respect to $\theta$ and $\tau$) are forced $\Gamma$-symmetric infinitesimally rigid.
\end{itemize}
\end{cor}

Note that (b) was already used in \cite{NS}.

As for incidental symmetry, for the reflection group $\mathcal{C}_s$ in $\mathbb{R}^d$, we also obtain the following complete analogue of Theorem~\ref{thm:regulartransfer}, whose proof is similar to Corollary \ref{cor:mirrortransfer}.

\begin{cor}\label{cor:mirrortransferforced}
Let $G=(V,E)$ be a graph,  $X\subseteq V$, and $\Gamma=\mathbb{Z}_2$. Further let $\tau(\mathbb{Z}_2)$ be the symmetry group $\mathcal{C}_s$ in $\mathbb{R}^d$. 
Then the following are equivalent:
\begin{itemize}
\item[(a)] $G$ can be realised as a forced $\Gamma$-symmetric infinitesimally rigid $\Gamma$-symmetric bar-joint framework on $\mathbb{S}^d$ (with respect to $\theta$ and $\tilde\tau$) such that the points assigned to $X$ lie on the equator, but not on the line through the origin that is perpendicular to the mirror hyperplane.
\item[(b)] $G$ can be realised as a  forced $\Gamma$-symmetric infinitesimally rigid $\Gamma$-symmetric point-hyperplane framework in $\mathbb{R}^d$ (with respect to $\theta$ and $\tau$) such that each vertex in $X$ is realised as a hyperplane,  no hyperplane is parallel to the mirror hyperplane, and each vertex in $V\setminus X$ is realised as a point.
\item[(c)]  $G$ can be realised as a  forced $\Gamma$-symmetric   infinitesimally rigid   $\Gamma$-symmetric bar-joint framework in $\mathbb{R}^d$ (with respect to $\theta$ and $\tau$) such that the points assigned to $X$ lie on a hyperplane (perpendicular to the mirror hyperplane).
\end{itemize}
\end{cor}

\section{Group pairings on $\mathbb{S}^2$ and in $\mathbb{R}^2$}
\label{sec:pairs}

We now consider relationships between symmetry groups with respect to infinitesimal rigidity and forced $\Gamma$-symmetric infinitesimal rigidity in both $\mathbb{S}^2$ and $\mathbb{R}^2$. (Analogous results for higher dimensions will be considered in Section~\ref{sec:higher}.) Throughout this section we will again assume that $G=(V,E)$ is a $\Gamma$-symmetric graph with respect to $\theta$, where $\theta$ acts freely on $V$. (Non-free actions are discussed in Section~\ref{sec:non-free}.)

 For simplicity we first deal with the basic pairing of mirror symmetry and half-turn symmetry. In later subsections we will generalise to other groups.

\subsection{Half-turn and mirror symmetry}
We prove that (forced $\mathbb{Z}_2$-symmetric) infinitesimal rigidity under half-turn symmetry is equivalent to (forced $\mathbb{Z}_2$-symmetric) infinitesimal rigidity under mirror symmetry on $\mathbb{S}^2$ (see also Figure~\ref{fig:transfersphere}.)

\begin{thm}\label{thm:mirrorhtpair} 
Let $G=(V,E)$ be a graph and let $\theta:\mathbb{Z}_2\to \textrm{Aut}(G)$ act freely on $V$. Further, let $X$ be a (possibly empty) subset of $V$. Then the following are equivalent:
\begin{itemize}
\item[(a)] $G$ can be realised as a $\mathbb{Z}_2$-symmetric (resp. forced $\mathbb{Z}_2$-symmetric) infinitesimally rigid bar-joint framework on $\mathbb{S}^2$ with respect to $\theta$ and $\tau:\mathbb{Z}_2\to \mathcal{C}_s$, where points assigned to $X$ lie on a great circle.
\item[(b)] $G$ can be realised as a $\mathbb{Z}_2$-symmetric (resp. forced $\mathbb{Z}_2$-symmetric) bar-joint framework on $\mathbb{S}^2$ with respect to $\theta$ and $\tau':\mathbb{Z}_2\to \mathcal{C}_2$, where points assigned to $X$ lie on a great circle.
\end{itemize}
\end{thm}


\begin{proof} We first prove the equivalence of (a) and (b) for infinitesimal rigidity.
Let $\mathbb{Z}_2=\{1,-1\}$. Suppose that $V_0=\{v_1,v_2,\dots,v_n\}$ is a set of representatives for the vertex orbits of $G$ under the action of $\theta$, and that $G$ has vertex set $\{v_1,v_1',v_2,v_2',\dots,v_n,v_n'\}$, with $\theta(-1)v_i=v_i'$ for all $i=1,\ldots, n$. Without loss of generality we consider $\tau(-1)$ to be the reflection in the plane $x=0$. Hence for a framework $(G,p)$ that is $\mathbb{Z}_2$-symmetric with respect to $\theta$ and $\tau$ we have $p(v_i)=(x_i,y_i,z_i)$ and $p(v_i')=(x_i',y_i',z_i')=(-x_i,y_i,z_i)$. Applying inversion to the set $I=V-V_0$ gives us $(x_i,y_i,z_i)$ for each $v_i\in V_0$ and $(x_i,-y_i,-z_i)$ for each $v_i'\in V-V_0$. Note that $(x_i,-y_i,-z_i)$ is the half-turn rotation of $(x_i,y_i,z_i)$ about the $x$-axis, so we let $\tau'(-1)$ be the half-turn rotation about the $x$-axis.  This partial inversion process is clearly  reversible, and since  inversion of points on $\mathbb{S}^2$ preserves infinitesimal rigidity by Lemma~\ref{lem:inv}, and since points on a great circle remain on the same great circle under inversion, the proof is complete.

Next we prove the equivalence of (a) and (b) for  forced $\mathbb{Z}_2$-symmetric infinitesimal rigidity.
Let  $(G,p)$ and $(G,q)$ be the two corresponding frameworks with $\mathcal{C}_s$ and $\mathcal{C}_2$ symmetry. The matrix  $O_{\theta,\tau}(G,p)$ corresponding to the linear system (\ref{eq:inner_inf}) and (\ref{eq:scale})  for  $(G,p)$ has the following form (this matrix is also known as the \emph{spherical orbit matrix} of $(G,p)$ \cite{BSWWSphere}):
\begin{displaymath} \bordermatrix{& & & & v_i & & & & v_j & & &
\cr & & & &  & & \vdots & &  & & &
\cr
 (v_i,v_j) & 0 & \ldots &  0 & (p_i-\tau(\psi_{ij})p_j) & 0 & \ldots & 0 &  (p_j-\tau(\psi_{ij})  p_i) &  0 &  \ldots&  0
 \cr & & & &  & & \vdots & &  & & &
\cr (v_i,v_i) & 0 & \ldots &  0 & 2(p_i-\tau(\psi_{ij})p_i) & 0 & \ldots & 0 &  0 &  0 &  \ldots&  0
 \cr & & & &  & & \vdots & &  & & &
\cr v_i & 0 & \ldots &  0 &p_{i}& 0 & \ldots & 0 & 0&  0 &  \ldots&  0
\cr & & & &  & & \vdots & &  & & &
\cr  v_j & 0 & \ldots &  0 & 0& 0 & \ldots & 0 & p_{j} &  0 &  \ldots&  0
\cr & & & &  & & \vdots & &  & & &
}
\textrm{,}\end{displaymath}
where $p_i=p(v_i)$ and $\psi_{ij}=\psi((v_i,v_j))$. (Note that $\tau(\psi_{ij})= \tau(\psi_{ij})^{-1}$ since $\tau(\psi_{ij})$ is an involution.)

We show that we can obtain the spherical orbit matrix $O_{\theta,\tau'}(G,q)$ for $(G,q)$ from $O_{\theta,\tau}(G,p)$ by carrying out elementary row operations.

For any edge $(v_i,v_j)$ with gain $\psi_{ij}=-1$, subtract the row corresponding to $v_i$ and the row corresponding to $v_j$ from the row corresponding to $(v_i,v_j)$. Subsequently, multiply the new row corresponding to $(v_i,v_j)$ by $-1$ and add back the rows corresponding to $v_i$ and $v_j$. Similarly, for any loop edge $(v_i,v_i)$ (which necessarily has the gain label $\psi_{ij}=-1$), we divide the row corresponding to $(v_i,v_i)$ by 2, then subtract the row corresponding to $v_i$, and then multiply the resulting row by -1. Finally we add back the row corresponding to $v_i$ and multiply the row by $2$. Any edge $(v_i,v_j)$ with gain $\psi_{ij}=1$ is left alone. The resulting matrix has the form
\begin{displaymath} \bordermatrix{& & & & v_i & & & & v_j & & &
\cr & & & &  & & \vdots & &  & & &
\cr
 (v_i,v_j) & 0 & \ldots &  0 & (p_i+\tau(\psi_{ij})p_j) & 0 & \ldots & 0 &  (p_j+\tau(\psi_{ij})  p_i) &  0 &  \ldots&  0
 \cr & & & &  & & \vdots & &  & & &
\cr (v_i,v_i) & 0 & \ldots &  0 & 2(p_i+\tau(\psi_{ij})p_i) & 0 & \ldots & 0 &  0 &  0 &  \ldots&  0
 \cr & & & &  & & \vdots & &  & & &
\cr v_i & 0 & \ldots &  0 &p_{i}& 0 & \ldots & 0 & 0&  0 &  \ldots&  0
\cr & & & &  & & \vdots & &  & & &
\cr  v_j & 0 & \ldots &  0 & 0& 0 & \ldots & 0 & p_{j} &  0 &  \ldots&  0
\cr & & & &  & & \vdots & &  & & &
}
\textrm{.}\end{displaymath}
By the definition of $\tau$ and $\tau'$, for each $i$ we have $\tau(-1)p_i= -\tau'(-1)p_i$. Thus, the matrix above is indeed equal to $O_{\theta,\tau'}(G,q)$, and  $O_{\theta,\tau}(G,p)$ and $O_{\theta,\tau'}(G,q)$ cleary have the same rank. By applying the above argument to a framework on the complete graph that affinely spans $\mathbb{R}^{3}$, we see that the spaces of trivial $\mathbb{Z}_2$-symmetric infinitesimal motions with respect to $\tau$ and $\tau'$ have the same dimension. This gives the result.
\end{proof}

\begin{figure}[htp]
        \begin{center} \includegraphics[scale=0.32]{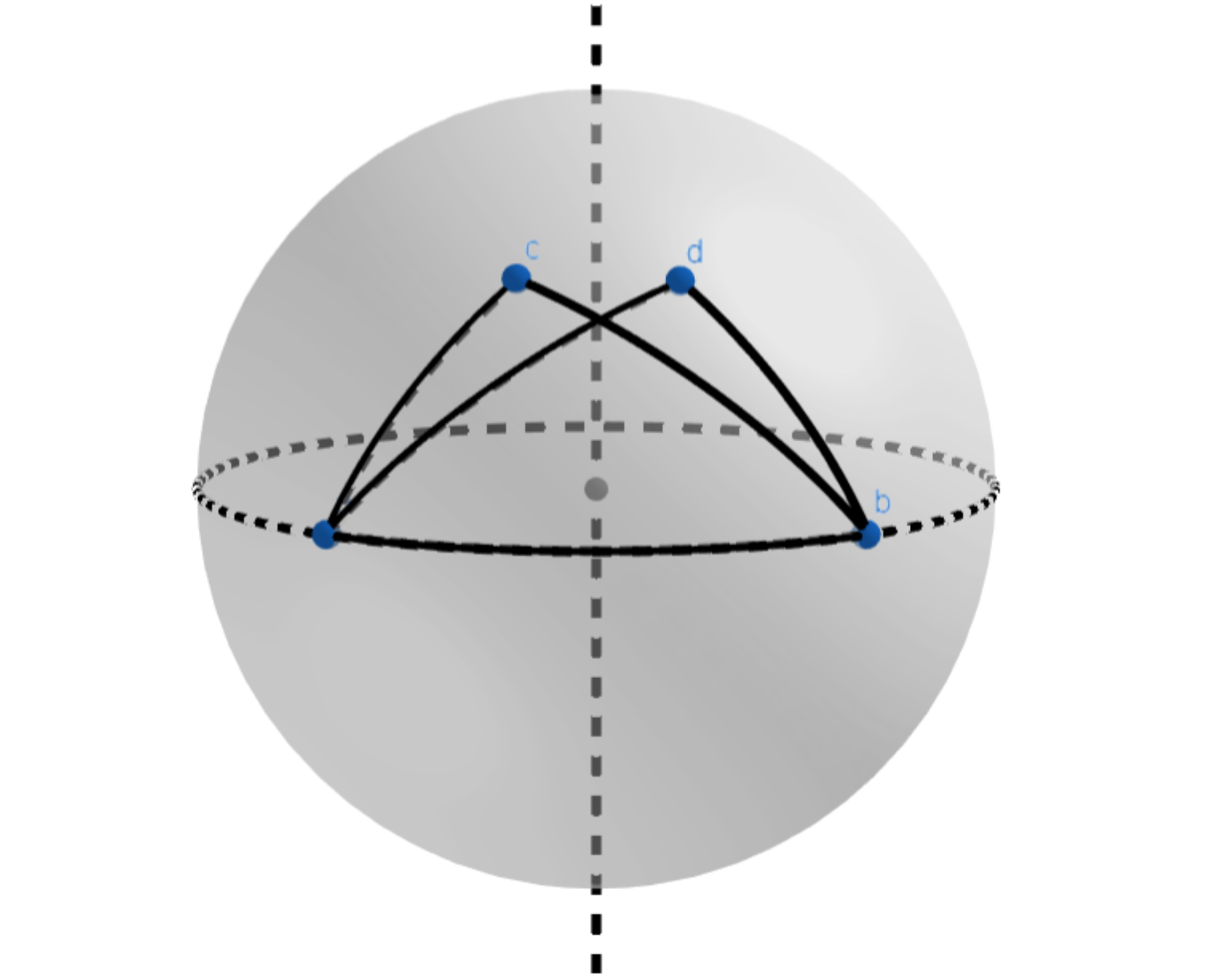} \hspace{0.6cm} 
         \includegraphics[scale=0.33]{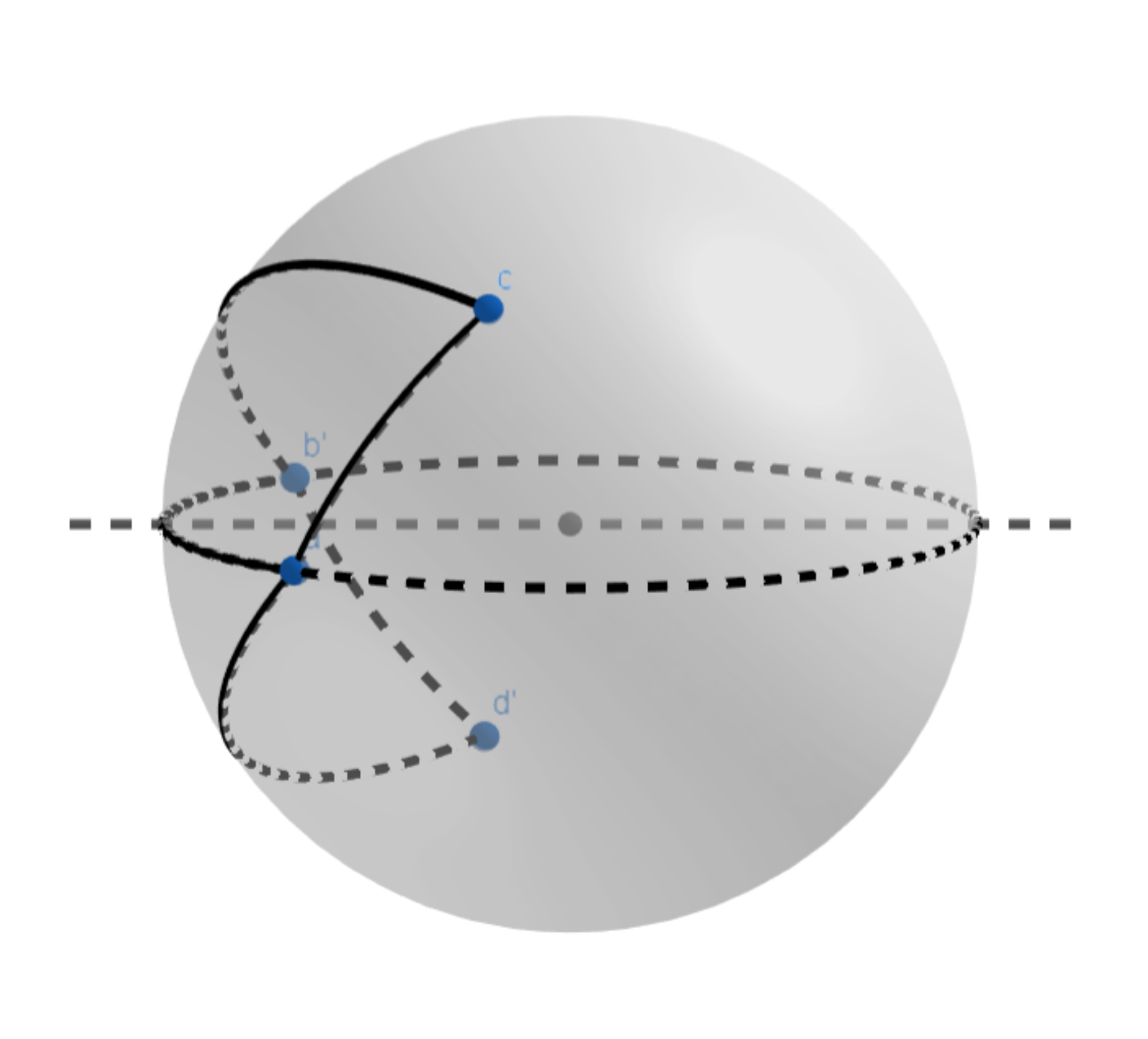}
        \end{center}   
\caption{Infinitesimally rigid frameworks on $\mathbb{S}^2$ with reflection and half-turn symmetry, illustrating the proof of Theorem~\ref{thm:mirrorhtpair}.}
\label{fig:transfersphere}
\end{figure}

From Theorem~\ref{thm:mirrorhtpair} we obtain the following corollary (see also Figure~\ref{fig:transferr2}).


\begin{cor}\label{cor:mirhalf} Let $G=(V,E)$ be a graph and let $\theta:\mathbb{Z}_2\to \textrm{Aut}(G)$ act freely on $V$.
Then the following are equivalent. 
 \begin{itemize}
\item[(a)] $G$ can be realised as a $\mathbb{Z}_2$-symmetric (resp. forced $\mathbb{Z}_2$-symmetric) infinitesimally rigid bar-joint framework in $\mathbb{R}^2$ with respect to $\theta$ and $\tau:\mathbb{Z}_2\to \mathcal{C}_s$.
\item[(b)]  $G$ can be realised as a $\mathbb{Z}_2$-symmetric (resp. forced $\mathbb{Z}_2$-symmetric) infinitesimally rigid bar-joint framework in $\mathbb{R}^2$ with respect to $\theta$ and $\tau':\mathbb{Z}_2\to \mathcal{C}_2$.
\end{itemize}
Moreover, for any nonempty subset $X$ of $V$, the following are equivalent.
 \begin{itemize}
\item[(c)] $G$ can be realised as a $\mathbb{Z}_2$-symmetric (resp. forced $\mathbb{Z}_2$-symmetric) infinitesimally rigid point-line framework in $\mathbb{R}^2$ with respect to $\theta$ and $\tau:\mathbb{Z}_2\to \mathcal{C}_s$, such that each vertex in $X$ is realised as a line and 
  each vertex in $V\setminus X$ is realised as a point. 
\item[(d)]  $G$ can be realised as a $\mathbb{Z}_2$-symmetric (resp. forced $\mathbb{Z}_2$-symmetric) infinitesimally rigid bar-joint framework in $\mathbb{R}^2$ with respect to $\theta$ and $\tau':\mathbb{Z}_2\to \mathcal{C}_2$, such that the points assigned to $X$ are collinear.
\end{itemize}
Finally, for any nonempty subset $X$ of $V$, the following are equivalent.
 \begin{itemize}
\item[(e)] $G$ can be realised as a $\mathbb{Z}_2$-symmetric (resp. forced $\mathbb{Z}_2$-symmetric) infinitesimally rigid bar-joint framework in $\mathbb{R}^2$ with respect to $\theta$ and $\tau:\mathbb{Z}_2\to \mathcal{C}_s$, such that each point assigned to $X$ lies on the mirror line (and is hence coincident with another point assigned to $X$).
\item[(f)]  $G$ can be realised as a $\mathbb{Z}_2$-symmetric (resp. forced $\mathbb{Z}_2$-symmetric) infinitesimally rigid point-line framework in $\mathbb{R}^2$ with respect to $\theta$ and $\tau':\mathbb{Z}_2\to \mathcal{C}_2$, such that each vertex in $X$ is realised as a line (and is hence parallel to another line assigned to $X$) and each vertex in $V\setminus X$ is realised as a point.
\end{itemize}
\end{cor}
\begin{proof} We first make some general remarks that are relevant to proving all three equivalences.

By Theorem~\ref{thm:mirrorhtpair}, there exists a $\mathbb{Z}_2$-symmetric (resp. forced $\mathbb{Z}_2$-symmetric) infinitesimally rigid bar-joint framework $(G,p)$ on $\mathbb{S}^2$ with respect to $\theta$ and $\tau:\mathbb{Z}_2\to \mathcal{C}_s$, where points assigned to $X$ lie on the equator if and only if  there exists a $\mathbb{Z}_2$-symmetric (resp. forced $\mathbb{Z}_2$-symmetric) infinitesimally rigid   $\mathbb{Z}_2$-symmetric bar-joint framework $(G,q)$ on $\mathbb{S}^2$ with respect to $\theta$ and $\tau':\mathbb{Z}_2\to \mathcal{C}_2$, where points assigned to $X$ lie on the equator. Let $(G,q)$ be obtained from $(G,p)$ as described in the proof of Theorem~\ref{thm:mirrorhtpair}. In particular, suppose (as in the proof of Theorem~\ref{thm:mirrorhtpair}) that  $\tau(-1)$ is the reflection in the $x=0$ plane and $\tau'(-1)$ is the half-turn around the $x$-axis. We now use the transfer mappings described in the proof of Theorem~\ref{thm:symtrans} to project these spherical frameworks to bar-joint or point-line frameworks in $\mathbb{R}^2$.

If necessary, we may invert orbits of points  of  $(G,p)$ (under the $\mathbb{Z}_2$-action) so that all points of the resulting framework $(G,i\circ p)$ lie on the (closed) upper hemisphere (preserving the mirror symmetry). We may then centrally project $(G,\iota \circ p)$ to the affine plane $z=1$ (which may then be identified with $\mathbb{R}^2$) to either obtain a bar-joint framework $(G,p')$ in $\mathbb{R}^2$ with $\mathcal{C}_s$ symmetry if no point of $(G,\iota \circ p)$ lies on the equator, or to obtain a point-line framework $(G,p',\ell')$  in $\mathbb{R}^2$ with $\mathcal{C}_s$ symmetry such that each vertex in $X$ is realised as a line and each vertex in $V\setminus X$ is realised as a point. 

Similarly, if necessary, we can invert orbis of points of $(G,q)$ (under the $\mathbb{Z}_2$-action) so that all points of the resulting framework $(G,i\circ q)$ lie on the (closed) left hemisphere (preserving the half-turn symmetry).  We may then rotate the entire framework $(G,i\circ q)$ about the $y$ axis by $\pi/2$ so that all points of the resulting framework $(G,\gamma \circ i\circ q)$ lie on the (closed) upper hemisphere. We may then centrally project $(G,\gamma \circ \iota \circ q)$ to the affine plane $z=1$ (which may then be identified with $\mathbb{R}^2$) to either obtain a bar-joint framework  $(G,q')$ in $\mathbb{R}^2$ with $\mathcal{C}_2$ symmetry if no point of $(G,\gamma \circ \iota \circ q)$ lies on the equator, or to obtain a point-line framework $(G,q',\ell')$ in $\mathbb{R}^2$  with $\mathcal{C}_2$ symmetry  such that each vertex in $X$ is realised as a line and each vertex in $V\setminus X$ is realised as a point. 

All the operations described above preserve infinitesimal rigidity and forced $\mathbb{Z}_2$-symmetric infinitesimal rigidity,  as shown in Sections~\ref{sec:transfer} and \ref{sec:forced}.

Note that $(G,p)$ has no point on the equator (or respectively the $x=0$ plane) if and only if $(G,q)$ has no point on the equator (resp. the $x=0$ plane). This proves the equivalence of (a) and (b).

For the equivalence of (c) and (d) we may assume (by slightly perturbing the vertices if necessary) that $(G,p)$ has no point on 
 the mirror-plane. If $(G,p)$ has a non-empty set of vertices positioned on the equator, then the same is true for $(G,q)$. Also, $(G,p)$ has no point on the $y$-axis if and only if $(G,q)$ has no point on the $y$-axis. Moreover, $(G,p)$ has no point on the mirror plane if and only if $(G,\gamma \circ \iota \circ q)$ has no point on the equator. So in this case, the operations described above for $(G,p)$ and $(G,q)$ yield the point-line framework $(G,p',\ell')$  in $\mathbb{R}^2$ with $\mathcal{C}_s$ symmetry and the bar-joint framework $(G,q')$ with $\mathcal{C}_2$ symmetry, respectively. This proves the equivalence of (c) and (d). 

Finally, note that $(G,p)$ has a coincident pair of points on the mirror plane if and only if $(G,\gamma \circ \iota \circ q)$ has a pair of opposite points on the equator. By rotating the entire framework $(G,\gamma \circ \iota \circ q)$ around its half-turn axis (i.e. the $z$-axis), we may always assume without loss of generality that $(G,\gamma \circ \iota \circ q)$ has no vertex on the $y$-axis, and hence neither does $(G,p)$. This proves the equivalence of (e) and (f).
\end{proof}



\begin{figure}[htp]
        \begin{center} \includegraphics[scale=0.5]{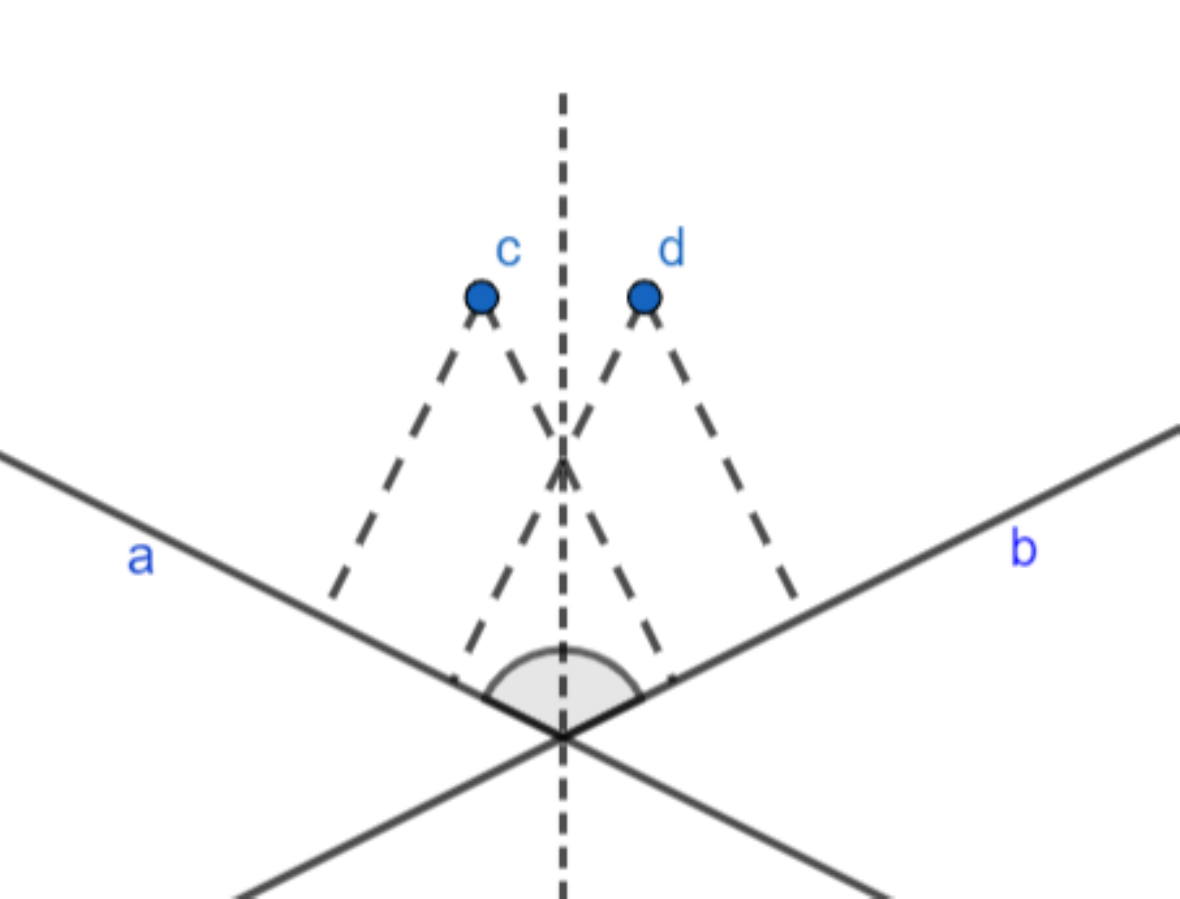} \hspace{0.5cm} 
         \includegraphics[scale=0.48]{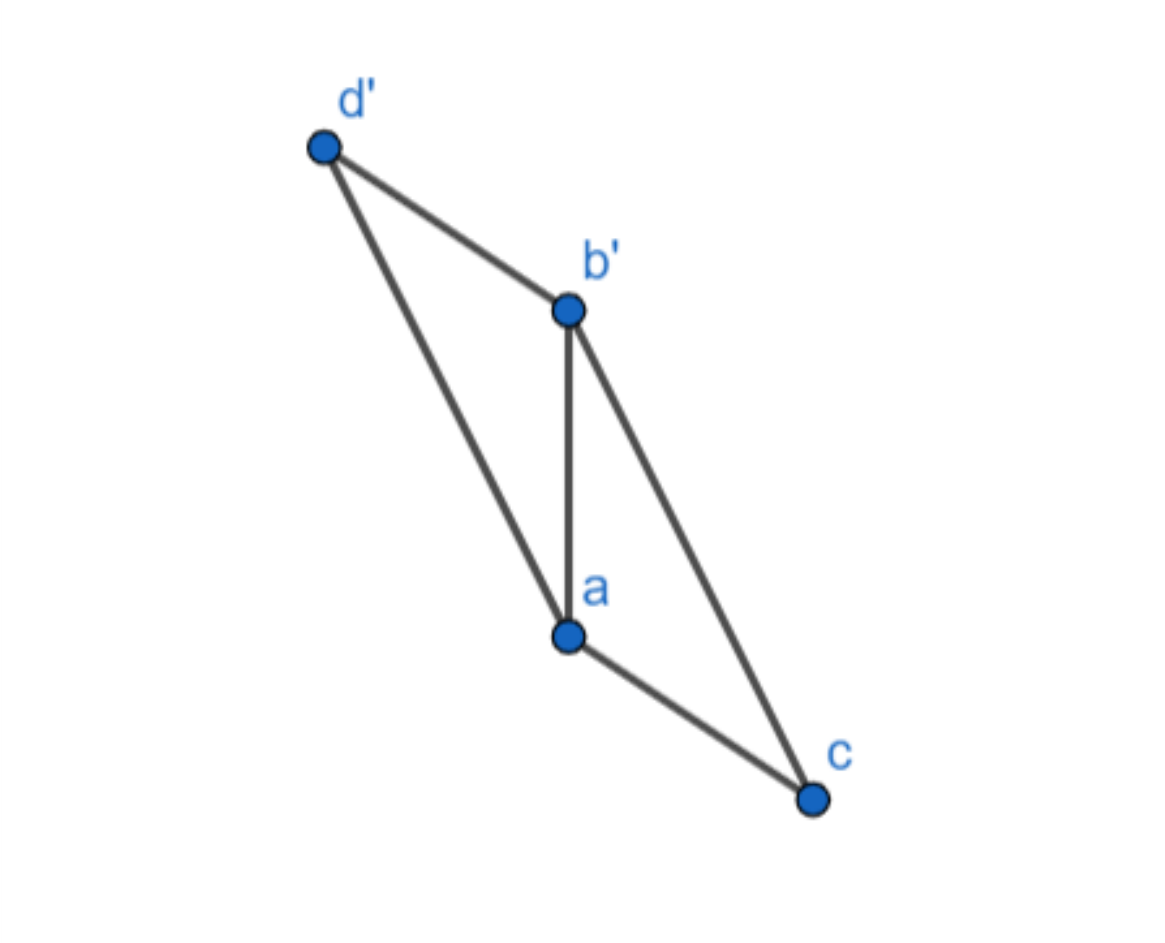}
        \end{center}   
\caption{The pair of infinitesimally rigid (point-line and bar-joint) frameworks in $\mathbb{R}^2$  corresponding to the pair of spherical frameworks shown in Figure~\ref{fig:transfersphere}, illustrating the proof of Corollary~\ref{cor:mirhalf}(c),(d).}
\label{fig:transferr2}
\end{figure}

\begin{rem} It follows from the equivalence of (a) and (b) in Corollary~\ref{cor:mirhalf} that a $\mathbb{Z}_2$-regular realisation of a graph $G$ as a bar-joint framework in $\mathbb{R}^2$ with respect to $\theta:\mathbb{Z}_2\to \textrm{Aut}(G)$ (which acts freely on $V$) and $\tau:\mathbb{Z}_2\to \mathcal{C}_s$ is infinitesimally rigid (resp. forced $\mathbb{Z}_2$-symmetric infinitesimally rigid) if and only if a $\mathbb{Z}_2$-regular realisation of $G$ as a bar-joint framework in $\mathbb{R}^2$ with respect to  $\theta$ and $\tau':\mathbb{Z}_2\to \mathcal{C}_2$ is infinitesimally rigid (resp. forced $\mathbb{Z}_2$-symmetric infinitesimally rigid), since we may use an argument similar to the one in Lemma~\ref{lem:regtran} to see that $\mathbb{Z}_2$-regularity is preserved under the described transfer.

 Therefore, mirror and half-turn symmetry have the same combinatorial characterisation for  $\mathbb{Z}_2$-regular infinitesimal rigidity (resp. forced $\mathbb{Z}_2$-symmetric infinitesimal rigidity) on $\mathbb{S}^2$, as well as in $\mathbb{R}^2$, by Corollaries~\ref{cor:combincidental} and \ref{cor:combforced}. While the characterisations for $\mathcal{C}_2$ and $\mathcal{C}_s$ are already known (recall Theorems~\ref{thm:forcedbjplane} and \ref{thm:combincidentalbj}), our result explains the combinatorics as more than a coincidence arising from the fact that corresponding spaces of trivial infinitesimal motions associated with the irreducible representations for $\mathcal{C}_2$ and $\mathcal{C}_s$ are the same. In particular, this equivalence of $\mathcal{C}_2$ and $\mathcal{C}_s$ does not rely on the `regularity' assumption.
\end{rem}

 Corollary~\ref{cor:mirhalf} may also be used to obtain the following combinatorial characterisation of (forced or incidentally) $\mathbb{Z}_2$-regular infinitesimally rigid point-line frameworks with exactly two lines and $\mathcal{C}_s$ symmetry (recall also Figure~\ref{fig:transfer_sph_ptline}).

\begin{thm}  \label{thm:lamanptline1} Let $\mathbb{Z}_2=\langle \gamma \rangle$ and let $(G,p,\ell)$ be a $\mathbb{Z}_2$-regular point-line framework in $\mathbb{R}^2$ with respect to $\theta:\mathbb{Z}_2 \to \textrm{Aut}_{PH}(G)$ and $\tau:\mathbb{Z}_2 \to \mathcal{C}_s$, where  $|V_H|=2$ and $\theta$ acts freely on $V=V_P\cup V_H$. Let $(G_0,\psi)$ be the quotient $\mathbb{Z}_2$-gain graph of $G$. Then $(G,p,\ell)$ is  infinitesimally rigid (resp. forced $\mathbb{Z}_2$-symmetric infinitesimally rigid) if and only if the quotient $\mathbb{Z}_2$-gain graph $(G_0,\psi)$ of $G$ contains a spanning $(2,3,i)$-gain-tight subgraph $(H_i,\psi_i)$ for each $i=1,2$ (resp. a spanning $(2,3,1)$-gain-tight subgraph $(H_1,\psi_1)$).
\end{thm}
\begin{proof} Suppose $(G,p',\ell')$ is  infinitesimally rigid (resp. forced $\mathbb{Z}_2$-symmetric infinitesimally rigid) with $\mathcal{C}_s$ symmetry as stated in the theorem. We transfer $(G,p',\ell')$ to a bar-joint framework $(G,q')$ in $\mathbb{R}^2$ with $\mathcal{C}_2$ symmetry as in the proof of Corollary~\ref{cor:mirhalf} (c),(d). 
Then $(G,q')$ is also  infinitesimally rigid (resp. forced $\mathbb{Z}_2$-symmetric infinitesimally rigid) and it follows from Theorem~\ref{thm:combincidentalbj} (resp. Theorem~\ref{thm:forcedbjplane}) that $(G_0,\psi)$ must satisfy the stated gain-sparsity conditions. 

Conversely, suppose that $(G_0,\psi)$ satisfies the stated gain-sparsity conditions. We claim that if $(G,p',\ell')$ is $\mathbb{Z}_2$-regular, then so  is $(G,q')$. To see this, note first that if $(G,p',\ell')$ is $\mathbb{Z}_2$-regular, then the spherical framework $(G,p)$ with $\mathcal{C}_s$ symmetry obtained from the central projection of $(G,p',\ell')$ is $\mathbb{Z}_2$-$V_H$-regular. In fact, since $|V_H|=2$ we may deduce that $(G,p)$ is even a $\mathbb{Z}_2$-regular spherical framework. In other words, we may slightly perturb the two points of $(G,p)$ corresponding to the vertices in $V_H$ in an arbitrary direction while preserving the $\mathcal{C}_s$ symmetry (in particular we may move them symmetrically off the equator) without reducing the rank of the corresponding spherical rigidity matrices. But this implies that there exists an open neighbourhood of $\mathbb{Z}_2$-symmetric bar-joint realisations of $G$ in $\mathbb{R}^2$ (with respect to $\theta$ and $\tau':\mathbb{Z}_2\to \mathcal{C}_2$) around $(G,q')$ in which the rank of the corresponding bar-joint rigidity matrices is maintained. Therefore, we may deduce that $(G,q')$ is indeed $\mathbb{Z}_2$-regular, as claimed.


 Thus, it follows from Theorem~\ref{thm:combincidentalbj} (resp. Theorem~\ref{thm:forcedbjplane}) that $(G,q')$ is  infinitesimally rigid (resp. forced $\mathbb{Z}_2$-symmetric infinitesimally rigid).   Therefore, the same is true for $(G,p',\ell')$ and the proof is complete. 
\end{proof}

\begin{rem} Since $\mathbb{Z}_2$-regularity is preserved under the transfer described in  Corollary~\ref{cor:mirhalf} (a),(b), and since any non-trivial 
$\mathbb{Z}_2$-symmetric infinitesimal motion extends to a non-trvial continuous motion for $\mathbb{Z}_2$-regular frameworks (see \cite{BSfinite,SWorbit} for details), it follows that we may also use Corollary~\ref{cor:mirhalf} (a),(b) to transfer \emph{continuous} motions between frameworks with $\mathcal{C}_2$ and $\mathcal{C}_s$ symmetry. Using the proof idea of \cite{BSfinite}, similar statements can also be obtained for the other transfers described in Corollary~\ref{cor:mirhalf}.
\end{rem}

\medskip

\subsection{All groups}


Theorem~\ref{thm:mirrorhtpair} can be extended to other pairings of groups, as follows.
As before, let $G=(V,E)$ be a $\Gamma$-symmetric graph with respect to $\theta:\Gamma\to \textrm{Aut}(G)$, where $\theta$ acts freely on $V$.

 The notation $\mathcal{G} \leftrightarrow \mathcal{H}$ for symmetry groups $\mathcal{G}$ and  $\mathcal{H}$ in dimension $3$ with the same underlying abstract group $\Gamma$ means that there exists a 
 $\Gamma$-symmetric  spherical framework $(G,p)$ on $\mathbb{S}^2$ with respect to $\theta$ and  $\tau(\Gamma)=\mathcal{G}$, and a $\Gamma$-symmetric  spherical framework $(G,q)$ on $\mathbb{S}^2$ with respect to $\theta$ and  $\tau'(\Gamma)=\mathcal{H}$ such that $(G,q)$ is obtained from $(G,p)$ by taking an index $2$ subgroup $\Gamma'$ of $\Gamma$ and inverting each point of $(G,p)$ assigned to the set  $V\setminus \{\gamma v:\, \gamma\in \Gamma', v\in V_0\}$, where $V_0$ is a set of representatives for the vertex orbits under the group action $\theta$.

\begin{thm}\label{thm:pairs}
Let $(G,p)$ be a $\Gamma$-symmetric framework on $\mathbb{S}^2$ with respect to $\theta$ and $\tau$ and let $(G,q)$ be a $\Gamma$-symmetric framework on $\mathbb{S}^2$ with respect to $\theta$ and $\tau'$ obtained from $(G,p)$ by the partial inversion process described above.
 Then $\tau(\Gamma) \leftrightarrow \tau'(\Gamma)$ must be one of the following pairings: 
\begin{itemize}
\item $\mathcal{C}_{2} \leftrightarrow \mathcal{C}_{s}$;
\item $\mathcal{C}_{2n} \leftrightarrow \mathcal{C}_{nh}$ where $n$ is odd;
\item $\mathcal{C}_{2n}\leftrightarrow \mathcal{S}_{2n}$ where $n$ is even;
\item $\mathcal{C}_{nv}\leftrightarrow \mathcal{D}_n$ for all $n$;
\item $\mathcal{C}_{2nv}\leftrightarrow \mathcal{D}_{nd}$ where $n$ is even;
\item $\mathcal{C}_{2nv}\leftrightarrow \mathcal{D}_{nh}$ where $n$ is odd;
\item $\mathcal{T}_d\leftrightarrow \mathcal{O}$.
\end{itemize}
Moreover, 
$(G,p)$ is a $\Gamma$-symmetric (resp. forced $\Gamma$-symmetric) infinitesimally rigid  framework on $\mathbb{S}^2$ with respect to $\theta$ and $\tau$ if and only if $(G,q)$ is a $\Gamma$-symmetric (resp. forced $\Gamma$-symmetric) infinitesimally rigid  framework on $\mathbb{S}^2$ with respect to $\theta$ and $\tau'$.
\end{thm}

\begin{proof}
The proof that the stated groups are linked and all possibilities are listed can be extracted from \cite{Con}. Alternatively it can be verified directly by applying the partial inversion mentioned above, as we illustrate in one particular case in Figure~\ref{fig:transferc6}.
We have already seen that inversion preserves  infinitesimal rigidity in Lemma \ref{lem:inv}. 
To prove the statement regarding forced $\Gamma$-symmetric infinitesmal rigidity, we may use exactly the same  argument (using the spherical orbit matrix) as in the proof of Theorem~\ref{thm:mirrorhtpair}. 
\end{proof}

\begin{figure}[htp]
        \begin{center} \includegraphics[scale=0.3]{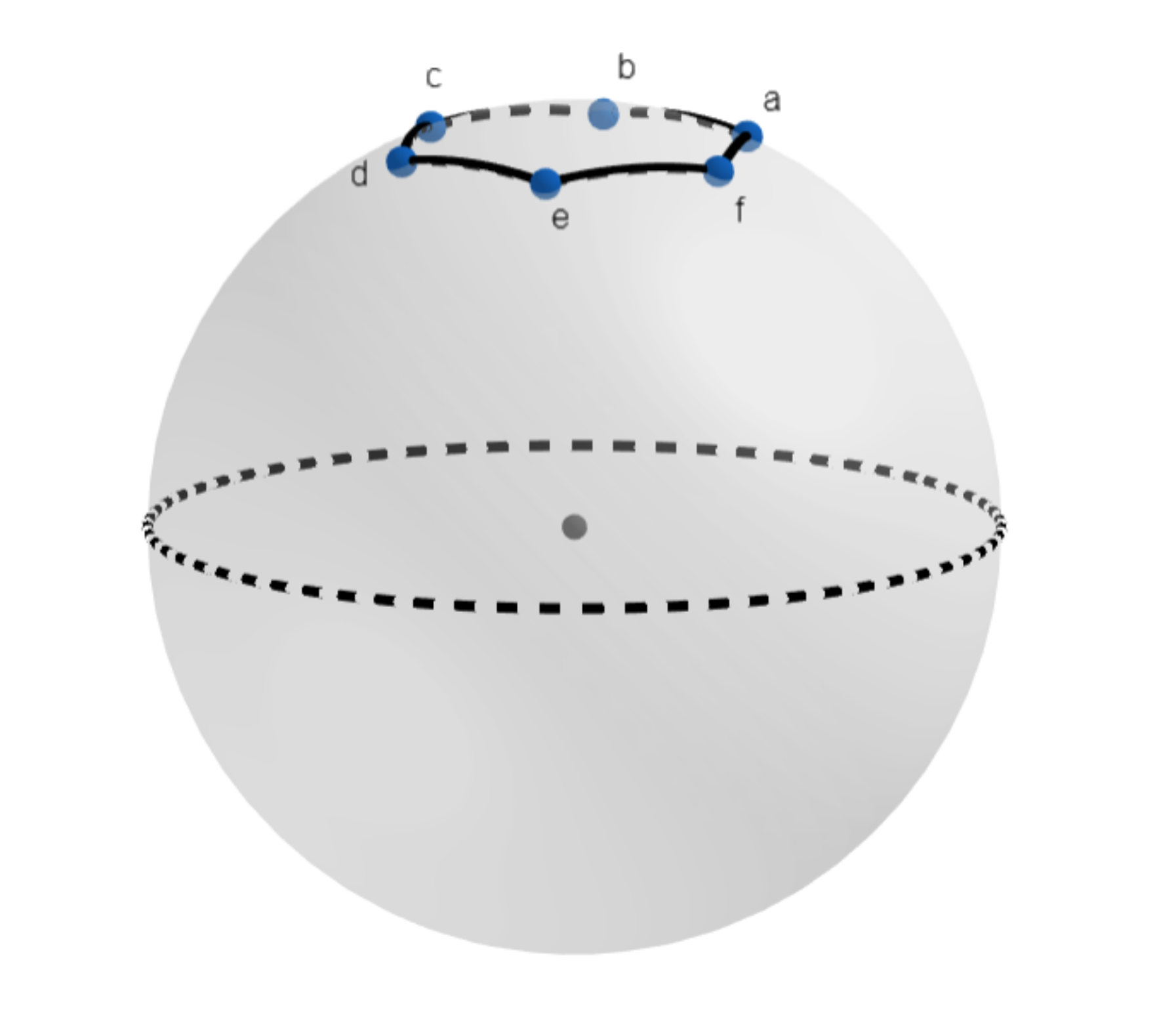} \hspace{0.6cm} 
        \includegraphics[scale=0.3]{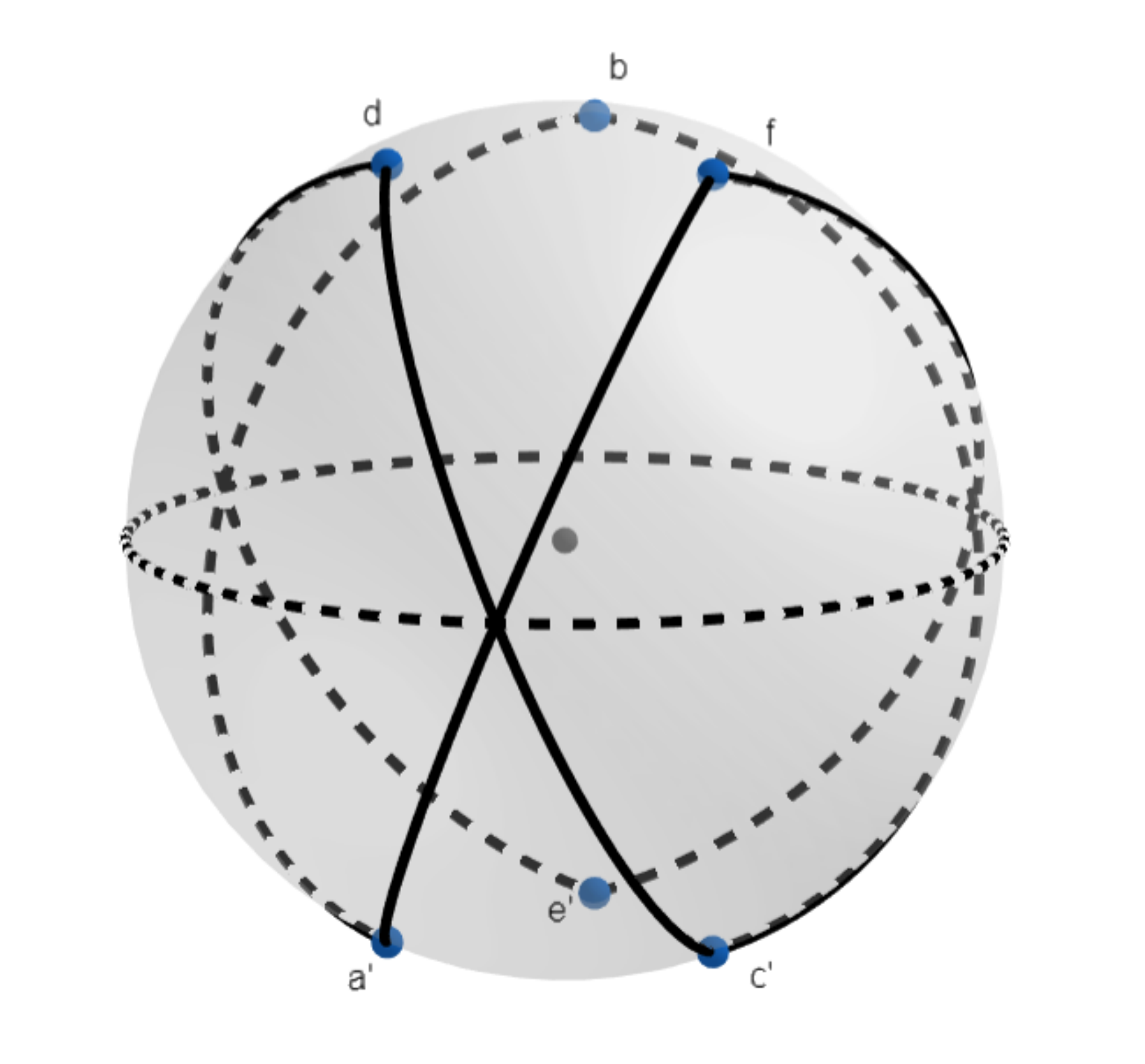}
        \end{center}   
\caption{A pair of frameworks on $\mathbb{S}^2$ with $\mathcal{C}_6$ and $\mathcal{C}_{3h}$ symmetry, illustrating Theorem~\ref{thm:pairs}.}
\label{fig:transferc6}
\end{figure}

Theorem~\ref{thm:pairs}    gives a complete classification of those symmetry groups which can be paired by partial inversion as described above. Every group not occurring in the statement either contains inversion as a group element ($\mathcal{C}_i$, $\mathcal{C}_{2nh}$, $\mathcal{D}_{2nh}$, $\mathcal{D}_{(2n+1)d}$, $\mathcal{S}_{2(2n+1)}$, $\mathcal{T}_h$, $\mathcal{O}_h$ and $\mathcal{I}_h$), or does not contain an index $2$ subgroup ($\mathcal{T}$ and $\mathcal{I}$), and hence no pairing would exist. (Note that for symmetry groups containing  inversion, the partial inversion process would not preserve the underlying abstract group.)

\subsection{Combinatorial consequences}

We obtain new combinatorial results and insights for the  $2$-sphere and the Euclidean plane from Theorem~\ref{thm:pairs}.

\subsubsection{Infinitesimal rigidity} Recall from Remark~\ref{rem1} that we currently only have combinatorial characterisations of $\Gamma$-regular infinitesimally rigid frameworks on $\mathbb{S}^2$ (or $\mathbb{R}^2$)  for the groups $\mathcal{C}_s$, $\mathcal{C}_2$ and $\mathcal{C}_n$, $n$ odd (where the action $\theta:\Gamma\to \textrm{Aut}(G)$ is free on the vertex set). It follows from Theorem~\ref{thm:pairs} that if we can extend these results for $\mathbb{S}^2$ to one of the groups listed in Theorem~\ref{thm:pairs}, then we immediately obtain the corresponding result for the paired group as a corollary. 

Moreover, if we manage to establish a combinatorial characterisation for  $\Gamma$-regular infinitesimal rigidity on $\mathbb{S}^2$, where $\tau(\Gamma)$ is a symmetry group of the form $\mathcal{C}_n$ or $\mathcal{C}_{nv}$ (for any $n\in \mathbb{N}$), then the central projection argument from Section~\ref{sec:transfer} immediately provides us with a combinatorial characterisation for  $\Gamma$-regular infinitesimal rigidity in $\mathbb{R}^2$, and vice versa. 

For symmetry groups in dimension 3 which do not exist in dimension 2 (such as  $\mathcal{C}_{nh}$, $\mathcal{S}_{2n}$, $\mathcal{D}_n$, etc.), this is not the case, since the central projection from the $2$-sphere to $\mathbb{A}^2$ (or  $\mathbb{R}^2$) would collapse the  group $\Gamma$ to a smaller group, and hence this process would generally yield a framework which is not $\Gamma'$-regular for the collapsed group $\Gamma'$. 

\subsubsection{Forced $\Gamma$-symmetric infinitesimal rigidity} As mentioned in the end of Section~\ref{subsec:sph},  combinatorial characterisations for $\Gamma$-regular \emph{forced} $\Gamma$-symmetric rigidity on $\mathbb{S}^2$ (where the action $\theta:\Gamma\to \textrm{Aut}(G)$ is free on the vertex set)  have been established for the groups $\mathcal{C}_s$, $\mathcal{C}_n$, $n\in \mathbb{N}$, $\mathcal{C}_i$, $\mathcal{C}_{nv}$, $n$ odd, $\mathcal{C}_{nh}$, $n$ odd, and $\mathcal{S}_{2n}$, $n$ even. The corresponding results for all other groups remain open, and some conjectures are given in \cite[Table 1]{NS}.

We can use the equivalence of forced $\Gamma$-symmetric infinitesimal rigidity for the pair $\mathcal{C}_{nv}\leftrightarrow \mathcal{D}_n$ given by Theorem \ref{thm:pairs} to deduce the following new result. We refer the reader to \cite[Def. 7.1]{jkt} for the definition of  a maximum $\mathcal{D}_n$-tight $\Gamma$-gain graph.

\begin{thm} 
Let $(G,p)$ be a $\Gamma$-regular framework on $\mathbb{S}^2$ with respect to $\theta$ and $\tau$, where $\tau(\Gamma)=\mathcal{D}_n$, $n$ odd. Let  $(G_0,\psi)$ be the quotient $\Gamma$-gain graph of $G$. Then $(G,p)$ is forced $\Gamma$-symmetric infinitesimally rigid if and only if $(G_0,\psi)$ contains a spanning subgraph that is maximum $\mathcal{D}_n$-tight.
\end{thm}

\begin{proof}
\cite[Theorem 8.2]{jkt} showed that the $\Gamma$-gain graph being maximum $\mathcal{D}_n$-tight is necessary and sufficient to characterise $\Gamma$-regular forced $\Gamma$-symmetric infinitesimal rigidity for $\mathcal{C}_{nv}$ in the case when $n$ is odd. 
Theorem \ref{thm:pairs} tells us that this is equivalent to $\Gamma$-regular forced $\Gamma$-symmetric infinitesimal rigidity for $\mathcal{D}_{n}$, giving the theorem.
\end{proof}

Note that the only symmetry groups for which we do not have combinatorial characterisations  for $\Gamma$-regular   forced $\Gamma$-symmetric infinitesimal rigidity in $\mathbb{R}^2$ are the groups $\mathcal{C}_{2nv}$, $n\in\mathbb{N}$, and significant new insights are needed to solve these cases.



\subsection{Double cover frameworks} 

In the previous section, we paired up symmetry groups with the same underlying abstract group $\Gamma$. Here we will see that some rigidity statements can still be developed without this condition.

The notation  $\mathcal{G} \twoheadrightarrow  \mathcal{H}$  for symmetry groups $\mathcal{G}$ and $\mathcal{H}$ in dimension $3$ with respective abstract groups $\Gamma$ and $\Gamma'$ means that there exists  a $\Gamma$-symmetric  spherical framework   on $\mathbb{S}^2$ with respect to $\theta$ and  $\tau(\Gamma)=\mathcal{G}$, where $\mathcal{G}$ does not contain the inversion element, and  a $\Gamma'$-symmetric  spherical framework $(G',p')$  on $\mathbb{S}^2$ with respect to $\theta'$ and  $\tau'(\Gamma')=\mathcal{H}$ such that $(G',p')$ is obtained from $(G,p)$ by taking the union of $(G,p)$ with the framework $(G,q)$ where $q$ is defined by $q(i)=-p(i)$ for each vertex $i$ of $G$.  Clearly, we have $2|\Gamma| = |\Gamma'|$. We say that $(G',p')$ is the \emph{double cover framework} of  $(G,p)$.


The most basic example is the pair  $\mathcal{C}_1 \twoheadrightarrow \mathcal{C}_i$, where $\mathcal{C}_1$ is  the trivial group and $\mathcal{C}_i$ is the inversion group in dimension 3. 

 The process of constructing the double cover framework $(G',p')$ of a spherical framework $(G,p)$ will clearly not preserve  infinitesimal rigidity since $(G',p')$ is disconnected, and hence will contain a $3$-dimensional space of non-trivial infinitesimal motions even when $(G,p)$ is infinitesimally rigid on $\mathbb{S}^2$. However, since none of these infinitesimal motions are $\Gamma'$-symmetric, this construction does preserve forced-symmetric infinitesimal rigidity. That is, we have the following result.

\begin{thm} 
Let $(G,p)$ be a $\Gamma$-symmetric  framework on $\mathbb{S}^2$ with respect to $\theta$ and $\tau$, and let $(G',p')$ be the $\Gamma'$-symmetric  double cover framework of $(G,p)$ on $\mathbb{S}^2$ with respect to $\theta'$ and $\tau'$. Then $\tau(\Gamma) \twoheadrightarrow \tau'(\Gamma')$ must be one of the following pairings.
\begin{itemize}
\item $\mathcal{C}_1 \twoheadrightarrow \mathcal{C}_i$;
\item $\mathcal{C}_s \twoheadrightarrow \mathcal{C}_{2h}$;
\item $\mathcal{C}_n\twoheadrightarrow \mathcal{S}_{2n}$, where $n$ is odd;
\item $\mathcal{C}_{n} \twoheadrightarrow \mathcal{C}_{nh}$, where $n$ is even;
\item $\mathcal{C}_{nv} \twoheadrightarrow \mathcal{D}_{nd}$, where $n$ is odd;
\item $\mathcal{C}_{nv} \twoheadrightarrow \mathcal{D}_{nh}$, where $n$ is even;
\item $\mathcal{C}_{nh} \twoheadrightarrow \mathcal{C}_{2nh}$, where $n$ is odd;
\item $\mathcal{S}_{2n} \twoheadrightarrow \mathcal{C}_{2nh}$ where $n$ is even;
\item $\mathcal{D}_{n} \twoheadrightarrow \mathcal{D}_{nd}$, where $n$ is odd;
\item $\mathcal{D}_{n} \twoheadrightarrow \mathcal{D}_{nh}$, where $n$ is even;
\item $\mathcal{D}_{nh} \twoheadrightarrow \mathcal{D}_{2nh}$, where $n$ is odd;
\item $\mathcal{D}_{nd} \twoheadrightarrow \mathcal{D}_{2nh}$, where $n$ is even;
\item $\mathcal{T} \twoheadrightarrow \mathcal{T}_{h}$; $\mathcal{T}_d \twoheadrightarrow \mathcal{O}_{h}$; $\mathcal{O} \twoheadrightarrow \mathcal{O}_{h}$; $\mathcal{I} \twoheadrightarrow \mathcal{I}_{h}$;
\end{itemize}
Moreover $(G,p)$ is forced $\Gamma$-symmetric infinitesimally rigid on $\mathbb{S}^2$ if and only if $(G',p')$ is forced $\Gamma'$-symmetric  infinitesimally rigid on $\mathbb{S}^2$.
\end{thm}

\begin{proof}
The groups $\tau(\Gamma)$ on the left hand side of the list of group pairings shown above are all the symmetry groups in dimension $3$ that do not contain the inversion element. 
 It is easy to check that the corresponding groups $\tau'(\Gamma')$ listed above satisfy $\tau(\Gamma) \twoheadrightarrow \tau'(\Gamma')$.

To see the final statement of the theorem, note that by definition, the double cover framework $(G',p')$ consists of two connected components with $\tau(\Gamma)$ symmetry, and the two components are images of each other under inversion. Any $\Gamma'$-symmetric infinitesimal motion of $(G',p')$ must preserve the (first order) distances between any pair of points lying in distinct components. 
This gives the result.
\end{proof}

Note that  transferring the result above to Euclidean space will result in the double cover frameworks having $|V(G')|/2|=|V(G)|$ pairs of coincident points. Hence when we have a combinatorial understanding of the smaller group, then the theorem gives us some combinatorial information about symmetric frameworks with pairs of coincident points. 
Note however that the general problem of characterising infinitesimal rigidity with a number of pairs of coincident points is likely to be challenging \cite{FJK}.

\section{Pairings in higher dimensions}
\label{sec:higher}

The result in Theorem \ref{thm:mirrorhtpair} can be easily generalised to higher dimensions. For simplicity, we will focus our discussion  on symmetry groups consisting of only involutions. Note that in our context an involution is an inversion in a $k$-dimensional subspace. 

Let $S$ be a $k$-dimensional subspace of $\mathbb{R}^d$ for some $k<d$. We  denote inversion in $S$ by $\iota_S$. So the matrix representing the isometry $\iota_S$ is the diagonal matrix with $1$'s corresponding to the ``dimensions of $S$'' and $-1$'s otherwise. Any involution $\iota_S$ clearly gives us a symmetry group  of order 2, which we denote by $\mathcal{C}_{\iota_S}$.




\begin{thm}\label{thm:inversionpair} 
Let $G=(V,E)$ be a graph and let $\theta:\mathbb{Z}_2\to \textrm{Aut}(G)$ act freely on $V$. Further, let $S_1,S_2$ be subspaces of $\mathbb{R}^{d+1}$ such that $\dim S_1=k_1, \dim S_2=k_2$, $k_1+k_2=d+1$ and $S_1\cap S_2$ is 0-dimensional, and let $X$ be a (possibly empty) subset of $V$. Then the following are equivalent:
\begin{itemize}
\item[(a)] $G$ can be realised as an infinitesimally rigid (resp. forced $\mathbb{Z}_2$-symmetric infinitesimally rigid)  $\mathbb{Z}_2$-symmetric bar-joint framework on $\mathbb{S}^d$ with respect to $\theta$ and $\tau:\mathbb{Z}_2\to \mathcal{C}_{\iota_{S_1}}$, where points assigned to $X$ lie on a great circle.
\item[(b)] $G$ can be realised as an infinitesimally rigid (resp. forced $\mathbb{Z}_2$-symmetric infinitesimally rigid) $\mathbb{Z}_2$-symmetric bar-joint framework on $\mathbb{S}^d$ with respect to $\theta$ and $\tau':\mathbb{Z}_2\to \mathcal{C}_{\iota_{S_2}}$, where points assigned to $X$ lie on a great circle.
\end{itemize}
\end{thm}
\begin{proof}
Let $\mathbb{Z}_2=\{1,-1\}$. Suppose that $V_0=\{v_1,v_2,\dots,v_n\}$ is a set of representatives for the vertex orbits of $G$ under the action of $\theta$, and that $G$ has vertex set $\{v_1,v_1',v_2,v_2',\dots,v_n,v_n'\}$, with $\theta(-1)v_i=v_i'$ for all $i=1,\ldots, n$. Let $(G,p)$ be a $\mathbb{Z}_2$-symmetric framework on $\mathbb{S}^d$ with respect to $\theta$ and $\tau:\mathbb{Z}_2\to \mathcal{C}_{\iota_{S_1}}$. We denote $p(v_i)=(x_{i_1}, x_{i_2},\ldots , x_{i_{d+1}})$ and we assume without loss of generality that $\tau(-1)p(v_i)=p(v_i')=(-x_{i_1}, -x_{i_2},\ldots , -x_{i_k}, x_{i_{k+1}}, \ldots, x_{i_{d+1}})$. Applying inversion to the set $I=V-V_0$ gives us $(x_{i_1}, x_{i_2},\ldots ,  x_{i_{d+1}})$ for each $v_i\in V_0$ and $(x_{i_1}, x_{i_2},\ldots , x_{i_k}, -x_{i_{k+1}}, \ldots, -x_{i_{d+1}})$ for each $v_i'\in V-V_0$. Note that $$ (x_{i_1}, x_{i_2},\ldots , x_{i_k}, -x_{i_{k+1}}, \ldots, -x_{i_{d+1}})=\tau'(-1)p(v_i).$$ It follows that $\mathcal{C}_{\iota_{S_1}}\leftrightarrow \mathcal{C}_{\iota_{S_2}}$. Since points on a great circle remain on the great circle under the partial inversion above, the proof is complete, by Lemma~\ref{lem:inv}.

The equivalence of (a) and (b) for forced $\mathbb{Z}_2$-symmetric infinitesimal rigidity follows in a similar manner (that is, via a sequence of row operations) to the proof of Theorem \ref{thm:mirrorhtpair}.
\end{proof}

Theorem~\ref{thm:inversionpair} in the case when $d=3$ gives us results which can be transferred to $\mathbb{R}^3$. In particular this yields a generalisation of the pairing between mirror symmetry and half-turn symmetry in the plane. In $3$-space, the corresponding pairing is mirror symmetry and inversion in a point, as the following result shows.


\begin{cor}\label{cor:mirinv3d} Let $G=(V,E)$ be a graph and let $\theta:\mathbb{Z}_2\to \textrm{Aut}(G)$ act freely on $V$.
Then the following are equivalent.
 \begin{itemize}
\item[(a)] $G$ can be realised as an infinitesimally rigid (resp. forced $\mathbb{Z}_2$-symmetric infinitesimally rigid) bar-joint framework in $\mathbb{R}^3$ with respect to $\theta$ and $\tau:\mathbb{Z}_2\to \mathcal{C}_s$.
\item[(b)]  $G$ can be realised as an infinitesimally rigid (resp. forced $\mathbb{Z}_2$-symmetric infinitesimally rigid) bar-joint framework in $\mathbb{R}^3$ with respect to $\theta$ and $\tau':\mathbb{Z}_2\to \mathcal{C}_i$.
\end{itemize}
\end{cor}

\begin{proof} The proof is analogous to the one for Corollary~\ref{cor:mirhalf}. By Theorem~\ref{thm:inversionpair}  there exist two infinitesimally rigid (resp. forced $\mathbb{Z}_2$-symmetric infinitesimally rigid) frameworks $(G,p)$ and $(G,q)$ on $\mathbb{S}^3$ which are $\mathbb{Z}_2$-symmetric  with respective symmetry groups $\mathcal{C}_s$ and $\mathcal{C}_{\iota_S}$, where $S$ is a line that is perpendicular to the mirror hyperplane  for the reflection in $\mathcal{C}_s$. We denote the coordinates of a point in $\mathbb{R}^4$ by $(x,y,w,z)$, and we suppose without loss of generality that the mirror hyperplane for the reflection in $\mathcal{C}_s$  is given by $x=0$, and that  $S$ is the $x$-axis. We may assume that $(G,p)$ has no point on the equator (i.e. on the hyperplane $z=0$) and no point on the mirror $x=0$. This is true if and only if $(G,q)$ also has no point on the $z=0$ or $x=0$ hyperplane. 

As described in the proof of Corollary~\ref{cor:mirhalf}, we may apply partial inversion to points of $(G,p)$ to move all points into the strict upper hemisphere (i.e. $z>0$ for all points), followed by a central projection of the resulting framework to the affine plane $z=1$ to obtain a bar-joint framework in $\mathbb{R}^3$ with symmetry group $\mathcal{C}_s$. 

Similarly, we may apply partial inversion to points of $(G,q)$ to move all points onto the strict left hemisphere (i.e. $x<0$ for all points), followed by a rotation of the whole framework by $\pi/2$ (taking the $x$-axis to the $z$-axis) to move all points onto the strict upper hemisphere. Central projection of the resulting framework to the affine plane $z=1$ then yields a bar-joint framework in $\mathbb{R}^3$ with symmetry group $\mathcal{C}_i$.

Since all of these operations preserve infinitesimal rigidity (resp. forced $\mathbb{Z}_2$-symmetric infinitesimal rigidity), the result follows.
\end{proof}

Similarly, we obtain the following result, which shows that $\mathcal{C}_{2v}$ and $\mathcal{C}_{2h}$ are also paired up in $\mathbb{R}^3$.

\begin{cor}\label{cor:c2vc2h3d} Let $G=(V,E)$ be a graph and let $\theta:\mathbb{Z}_2\times \mathbb{Z}_2\to \textrm{Aut}(G)$ act freely on $V$.
Then the following are equivalent.
 \begin{itemize}
\item[(a)] $G$ can be realised as an infinitesimally rigid (resp. forced $\mathbb{Z}_2\times \mathbb{Z}_2$-symmetric infinitesimally rigid) bar-joint framework in $\mathbb{R}^3$ with respect to $\theta$ and $\tau:\mathbb{Z}_2\times \mathbb{Z}_2\to \mathcal{C}_{2v}$.
\item[(b)]  $G$ can be realised as an infinitesimally rigid (resp. forced $\mathbb{Z}_2\times \mathbb{Z}_2$-symmetric infinitesimally rigid) bar-joint framework in $\mathbb{R}^3$ with respect to $\theta$ and $\tau':\mathbb{Z}_2\times \mathbb{Z}_2\to \mathcal{C}_{2h}$.
\end{itemize}
\end{cor}

\begin{proof} First, we claim that on $\mathbb{S}^3$, the symmetry group $\mathcal{C}_{2v}$ with the two reflections having the respective mirror hyperplanes $x=0$ and $y=0$,  and the half-turn having the $2$-dimensional axis $x=y=0$, is paired with the symmetry group $\mathcal{S}$ generated by the half-turn with the axis $w=z=0$ and the generalised inverson in the $y$-axis.

To see this, consider a vertex orbit of a $\mathbb{Z}_2\times \mathbb{Z}_2$-symmetric framework $(G,p)$ on $\mathbb{S}^3$ with respect to $\theta$ and $\tilde{\tau}:\mathbb{Z}_2\times \mathbb{Z}_2\to \mathcal{C}_{2v}$. The points of $(G,p)$ corresponding to such an orbit are of the form $(x,y,w,z)$, $(-x,y,w,z)$, $(-x,-y,w,z)$, $(x,-y,w,z)$. Now, invert the two points corresponding to the orbit of $(x,y,w,z)$ under the reflection in the $y=0$ hyperplane (i.e., the points $(x,y,w,z)$ and $(-x,y,-w,-z)$)  to obtain  the orbit $(-x,-y,-w,-z)$, $(-x,y,w,z)$, $(-x,-y,w,z)$, $(-x,y,-w,-z)$. If we do this for each vertex orbit of $(G,p)$, then we obtain a $\mathbb{Z}_2\times \mathbb{Z}_2$-symmetric framework $(G,q)$ on $\mathbb{S}^3$ with respect to $\theta$ and $\tilde{\tau}':\mathbb{Z}_2\times \mathbb{Z}_2\to \mathcal{S}$, as claimed.

Consider the spherical framwork $(G,p)$ with $\mathcal{C}_{2v}$ symmetry. Invert orbits of  points of $(G,p)$ to move all points onto the strict upper hemishere ($z>0$) and then centrally project the resulting framework onto the hyperplane $z=1$ to obtain a $\mathbb{Z}_2\times \mathbb{Z}_2$-symmetric  bar-joint framework in $\mathbb{R}^3$ with respect to $\theta$ and $\tau:\mathbb{Z}_2\times \mathbb{Z}_2\to \mathcal{C}_{2v}$. (The points corresponding to a vertex orbit of this framework have the form  $(x,y,w)$, $(-x,y,w)$, $(-x,-y,w)$, $(x,-y,w)$).

Consider the other spherical framework $(G,q)$ with $\mathcal{S}$ symmetry. Rotate the framework by $\pi/2$ taking the $x$-axis to the $z$-axis (i.e., the rotation matrix has a $1$ in positions $(1,4), (2,2), (3,3)$ and a $-1$ in position $(4,1)$, and zeros elsewhere) to move all points onto the strict upper hemisphere. Then the points of any vertex orbit of the resulting framework are of the form $(-z,-y,-w,x)$, $(z,y,w,x)$, $(z,-y,w,x)$, $(-z,y,-w,x)$. Finally centrally project the framework onto the hyperplane $z=1$ to obtain a $\mathbb{Z}_2\times \mathbb{Z}_2$-symmetric  bar-joint framework in $\mathbb{R}^3$ with respect to $\theta$ and $\tau':\mathbb{Z}_2\times \mathbb{Z}_2\to \mathcal{C}_{2h}$, where $\mathcal{C}_{2h}$ is generated by the reflection with mirror plane $y=0$ and the half-turn about the $y$-axis. (The points corresponding to a vertex orbit of this framework have the form  $(-z,-y,-w)$, $(z,y,w)$, $(z,-y,w)$, $(-z,y,-w)$).

Since all of these operations preserve infinitesimal rigidity (resp. forced $\mathbb{Z}_2\times \mathbb{Z}_2$-symmetric infinitesimal rigidity), the result follows.
\end{proof}

\begin{rem} From the proofs of Corollaries~\ref{cor:mirinv3d} and \ref{cor:c2vc2h3d}  we can also easily obtain analogous statements to the ones in Corollary~\ref{cor:mirhalf}(c),(d) and Corollary~\ref{cor:mirhalf}(e),(f) for the group pairings $\mathcal{C}_s$, $\mathcal{C}_i$, and $\mathcal{C}_{2v}$, $\mathcal{C}_{2h}$. We leave the details to the reader.
\end{rem}

The only other symmetry groups  containing only involutions in $\mathbb{R}^3$ are $\mathcal{C}_2$ and $\mathcal{D}_2$. Neither of them is paired with another group. 

\begin{rem} It follows from  Corollary~\ref{cor:mirinv3d} that a $\mathbb{Z}_2$-regular realisation of a graph $G$ as a bar-joint framework in $\mathbb{R}^3$ with respect to $\theta:\mathbb{Z}_2\to \textrm{Aut}(G)$ (which acts freely on $V$) and $\tau:\mathbb{Z}_2\to \mathcal{C}_s$ is infinitesimally rigid (resp. forced $\mathbb{Z}_2$-symmetric infinitesimally rigid) if and only if a $\mathbb{Z}_2$-regular realisation of $G$ as a bar-joint framework in $\mathbb{R}^2$ with respect to  $\theta$ and $\tau':\mathbb{Z}_2\to \mathcal{C}_i$ is infinitesimally rigid (resp. forced $\mathbb{Z}_2$-symmetric infinitesimally rigid), since $\mathbb{Z}_2$-regularity is preserved under the described transfer.

In particular, this provides a direct geometric argument for the fact that the combinatorial characterisations for $\mathbb{Z}_2$-regular infinitesimal rigidity (resp. forced $\mathbb{Z}_2$-symmetric infinitesimally rigid) are the same for \emph{body-bar} frameworks (i.e. structures consisting of full-dimensional rigid bodies connected in pairs by rigid bars) with mirror and inversion symmetry in $\mathbb{R}^3$, as shown in \cite{STbb}. (See also \cite{GSW}.)

Similarly, Corollary~\ref{cor:c2vc2h3d} explains  the fact that the combinatorial characterisations for $\mathbb{Z}_2\times \mathbb{Z}_2$-regular infinitesimal rigidity (resp. forced $\mathbb{Z}_2\times \mathbb{Z}_2$-symmetric infinitesimal rigidity)  are the same for body-bar frameworks with $\mathcal{C}_{2v}$ and $\mathcal{C}_{2h}$ symmetry in $\mathbb{R}^3$.
\end{rem}

Note that  combinatorial characterisations of $\Gamma$-regular forced $\Gamma$-symmetric infinitesimally rigid body-bar frameworks have been established for all symmetry groups in general dimension \cite{Tan15}. Moreover, combinatorial characterisations of $\Gamma$-regular infinitesimally rigid body-bar frameworks  have been established for all symmetry groups that have $\mathbb{Z}_2\times \cdots \times \mathbb{Z}_2$ as an underlying abstract group $\Gamma$ \cite{STbb}. Thus, our new geometric insights do not yield any new combinatorial results regarding the rigidity of symmetric body-bar frameworks in $\mathbb{R}^d$.

There are of course many symmetry groups in $\mathbb{R}^d$, $d\geq 3$, that contain elements that are not involutions. We leave it as an open problem to establish a complete list of group pairings in these higher-dimensional spaces.


\section{Non-free actions}
\label{sec:non-free}

In Section~\ref{sec:pairs}, we made the assumption that $G=(V,E)$ is a $\Gamma$-symmetric graph with respect to $\theta:\Gamma\to \textrm{Aut}(G)$, where $\theta$ acts freely on $V$. In this section we will now consider the case where $\theta$ may fix some vertices of $G$. 

\subsection{Background}

We say that a vertex $i$ of  a $\Gamma$-symmetric graph $G$ (with respect to $\theta$) is \textrm{fixed} by $\gamma\in \Gamma$, $\gamma\neq 1$, if  $\theta(\gamma)(i)=i$ (or in short $\gamma (i)=i$). Similarly, an edge $\{i,j\}$ is fixed by $\gamma$ if either $\gamma(i)=i$ and $\gamma(j)=j$ or $\gamma(i)=j$ and $\gamma(j)=i$.
The number of vertices and edges of $G$ that are fixed by $\gamma$ are denoted by $|V_\gamma|$ and $|E_\gamma|$, respectively.

For forced $\Gamma$-symmetric rigidity, an orbit matrix has been established for bar-joint frameworks in \cite{SWorbit} which allows for $\theta$ to be non-free on $V$. However, the structure of the orbit matrix becomes significantly more complex when $\theta$ is not free on $V$ and hence the corresponding conditions for  forced $\Gamma$-symmetric rigidity also become more involved.  Thus, the combinatorics of $\Gamma$-regular forced $\Gamma$-symmetric rigidity has not yet been properly investigated in the case when $\theta$ is not free on $V$. In the following we will therefore focus on incidentally symmetric frameworks.

A (bar-joint, spherical or point-hyperplane) framework is called  \emph{isostatic} if it is infinitesimally rigid and the removal of any edge yields an infinitesimally flexible framework. 
For bar-joint frameworks,  the following combinatorial characterisations of $\Gamma$-regular isostatic  frameworks in the plane were established in \cite{SchC3,Sch}. 

We say that a graph $G=(V,E)$ is \emph{$(2,3)$-tight} if $|E|=2|V|-3$ and for all subgraphs $(V',E')$ with $|E'|>0$ we have $|E'|\leq 2|V'|-3$.

\begin{thm}[\cite{SchC3,Sch}] \label{thm:bernd} 
Let $\Gamma=\langle \gamma \rangle$ and let $(G,p)$ be a $\Gamma$-regular bar-joint framework (with respect to $\theta$ and $\tau$) in $\mathbb{R}^2$, where $\tau(\Gamma) \in \{ \mathcal{C}_s, \mathcal{C}_2,\mathcal{C}_3\}$. Then $(G,p)$ is isostatic if and only if $G$ is $(2,3)$-tight and
\begin{itemize}
\item[(i)] $|E_\gamma|=1$ for $\tau(\Gamma)=\mathcal{C}_s$.
\item[(ii)] $|V_\gamma|=0$ and $|E_\gamma|=1$ for $\tau(\Gamma)= \mathcal{C}_2$.
\item[(iii)] $|V_\gamma|=0$ for $\tau(\Gamma)= \mathcal{C}_3$.
\end{itemize}
\end{thm}

As shown in  \cite{Cetal}, only two other non-trivial symmetry groups can give isostatic frameworks in the plane, namely $\mathcal{C}_{2v}$ and $\mathcal{C}_{3v}$. A Laman-type theorem (analogous to the one above) has not yet been established for these groups, see \cite{SchC3,Sch}. There are also no combinatorial characterisations of $\Gamma$-regular isostatic frameworks in higher dimensions, except that for \emph{body-bar} frameworks in Euclidean 3-space, some partial results, as well as a number of conjectures, were given \cite{GSW}, for a range of symmetry groups.  

 Combinatorial characterisations analogous to the ones in  Theorem~\ref{thm:bernd} have not yet been investigated for symmetric spherical frameworks or point-hyperplane frameworks. (Necessary conditions for a symmetric point-line framework in the plane to be isostatic have been obtained in \cite{OP}, but to the best of our knowledge, sufficiency of these conditions has not yet been investigated.)
 
 Note that Corollary~\ref{cor:mirhalf} explains why the conditions in Theorem~\ref{thm:bernd} are the same for  $\mathcal{C}_2$ and $\mathcal{C}_s$ in the case when $\theta$ acts freely on $V$. Moreover, it immediately follows from the results in Section~\ref{sec:transfer} that Theorem~\ref {thm:bernd} also gives combinatorial characterisations of $\Gamma$-regular isostatic \emph{spherical} frameworks on $\mathbb{S}^2$ with mirror, half-turn and 3-fold rotational  symmetry. 

\subsection{Group pairings under non-free actions}\label{subsec:nonfreepair}

Consider our partial inversion process of linking up groups on the $d$-sphere where some vertices are fixed by non-trivial group elements. Here if the fixed vertices are left alone in the partial inversion, then the resulting framework will typically not be symmetric. We will discuss this issue further in Section~\ref{sec:elliptic}. However, in the projection to Euclidean $d$-space, symmetry can be recovered.

The key example is the $\mathcal{C}_2$ and $\mathcal{C}_s$ pairing on $\mathbb{S}^2$ and $\mathbb{R}^2$.
Take a bar-joint framework with $\mathcal{C}_s$ symmetry in $\mathbb{R}^2$ and  project it (as described in the proof of Corollary~\ref{cor:mirhalf}) to a spherical framework $(G,p)$ with $\mathcal{C}_s$ symmetry on $\mathbb{S}^2$. Suppose $G$ has a vertex $v$ that is fixed by the reflection. Now apply the partial inversion to all vertex orbits of size 2 of $(G,p)$, as described  in the proof of Theorem~\ref{thm:mirrorhtpair}. Since $v$ is in a vertex orbit of size 1, it is left alone in the partial inversion process, so the resulting spherical framework $(G,q)$ does not have $\mathcal{C}_2$ symmetry (unless we add the symmetric copy of the point corresponding to $v$; see Section~\ref{sec:elliptic}). However, when we project $(G,q)$ to a point-line framework in $\mathbb{R}^2$ as described in the proof of Corollary~\ref{cor:mirhalf}, then the point corresponding to the  vertex $v$ is mapped to a line, and we may assume that it goes through the origin (since lines can be translated without affecting infinitesimal rigidity, by Remark~\ref{rem:shift}). Thus, the resulting point-line framework in $\mathbb{R}^2$ does have the desired $\mathcal{C}_2$ symmetry. This yields the following extension of Corollary~\ref{cor:mirhalf}.

\begin{cor}\label{cor:mirhalffixed} Let $G=(V,E)$ be a $\mathbb{Z}_2$-symmetric graph with respect to $\theta:\mathbb{Z}_2\to \textrm{Aut}(G)$, and let $F$ be the subset of vertices of $G$ that are fixed by the non-trivial element of $\mathbb{Z}_2$ (with respect to $\theta$). Then the following are equivalent. 
 \begin{itemize}
\item[(a)] $G$ can be realised as a $\mathbb{Z}_2$-symmetric isostatic bar-joint framework in $\mathbb{R}^2$ with respect to $\theta$ and $\tau:\mathbb{Z}_2\to \mathcal{C}_s$.
\item[(b)]  $G$ can be realised as a $\mathbb{Z}_2$-symmetric isostatic point-line framework in $\mathbb{R}^2$ with respect to $\theta$ and $\tau':\mathbb{Z}_2\to \mathcal{C}_2$, such that each vertex in $F$ is realised as a line and each vertex in $V\setminus F$ is realised as a point.
\end{itemize}
\end{cor}

\begin{figure}[htp]
        \begin{center} \includegraphics[scale=0.5]{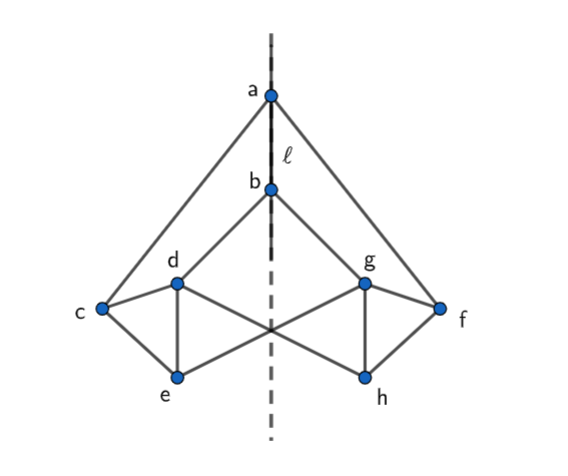} \hspace{0.3cm} 
         \includegraphics[scale=0.5]{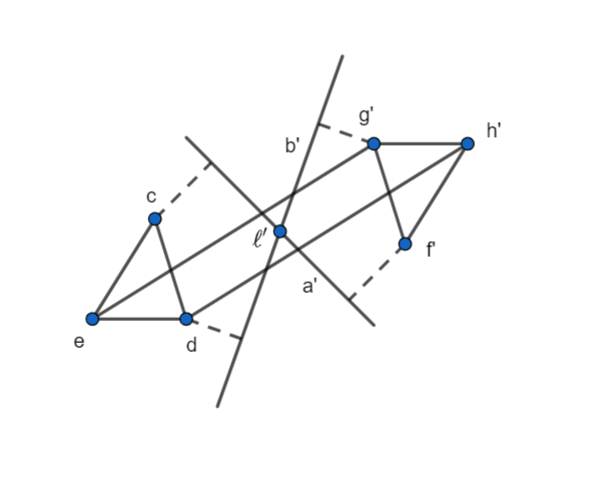}
        \end{center}   
\caption{A point-line framework $(G,p,\ell)$ with $\mathcal{C}_s$ symmetry in $\mathbb{R}^2$. This framework  has a non-trivial symmetry-preserving motion and gives a symmetric point-line framework model of the `grab-bucket mechanism' in engineering (\cite{RosenWills} p. 270).  It has the two point vertices $a,b$ and the line vertex $\ell$ fixed by the reflection, and the line corresponding to $\ell$ lies along the mirror line. $(G,p,\ell)$ may be transformed to the flexible point-line framework with $\mathcal{C}_2$ symmetry in $\mathbb{R}^2$ shown on the right. }
\label{fig:bucket}
\end{figure}

More generally, by carefully tracking the effect of our transfer mappings on points and lines that are fixed by a reflection or half-turn, we obtain the following result. (See also Figure~\ref{fig:bucket} for an illustration of an example).

\begin{cor}\label{cor:mirhalffixed2} Let $G=(V,E)$ be a $\mathbb{Z}_2$-symmetric graph with respect to $\theta:\mathbb{Z}_2\to \textrm{Aut}(G)$, and suppose $F$ is the subset of vertices of $G$ that are fixed by the non-trivial element of $\mathbb{Z}_2$ (with respect to $\theta$). Let $X$ be a non-empty subset of $V$, and let $F_X=F\cap X$ and $F'_{X}=F\setminus X$.  Then the following are equivalent. 
 \begin{itemize}
\item[(a)] $G$ can be realised as a $\mathbb{Z}_2$-symmetric isostatic point-line framework in $\mathbb{R}^2$ with respect to $\theta$ and $\tau:\mathbb{Z}_2\to \mathcal{C}_s$ such that each vertex in $X$ is realised as a line (with $F_{X_{\parallel}}$ and $F_{X_{\perp}}$ denoting the vertices in $F_X$ that are realised parallel and perpendicular to the mirror line of the reflection in $\mathcal{C}_s$, respectively) and each vertex in $V\setminus X$ is realised as a point.
\item[(b)]  $G$ can be realised as a $\mathbb{Z}_2$-symmetric isostatic point-line framework in $\mathbb{R}^2$ with respect to $\theta$ and $\tau':\mathbb{Z}_2\to \mathcal{C}_2$, such that each vertex in $F'_{X}$ and each vertex in $F_{X_{\perp}}$ is realised as a line (with all lines in $F_{X_{\perp}}$ parallel to each other), and each vertex in $V\setminus (F'_{X}\cup F_{X_{\perp}})$ is realised as a point so that all the points of  $X\setminus F_{X_{\perp}}$ are collinear (and perpendicular to the lines  of $F_{X_{\perp}}$) and all points of $F_{X_{\parallel}}$ lie at the origin.
\end{itemize}
\end{cor}

\begin{proof} Let $(G,p,\ell)$ be a $\mathbb{Z}_2$-symmetric isostatic point-line framework in $\mathbb{R}^2$ with respect to $\theta$ and $\tau:\mathbb{Z}_2\to \mathcal{C}_s$ and apply to it the transfer mappings from the proofs of Theorem~\ref{thm:mirrorhtpair} and Corollary~\ref{cor:mirhalf}. In this transfer any point of $(G,p,\ell)$ corresponding to a vertex in $F'_X$ is mapped to a line,  all lines of $(G,p,\ell)$ corresponding to a vertex in $F_{X_{\perp}}$ are mapped to parallel lines, and all lines of $(G,p,\ell)$ corresponding to a vertex in $F_{X_{\parallel}}$ are mapped to a point at the origin. Moreover, the lines of $(G,p,\ell)$ corresponding to vertices in $X\setminus F$ are mapped to collinear points (with the line of collinearity being perpendicular to the lines of $F_{X_{\perp}}$), and any point of $(G,p,\ell)$ corresponding to a vertex in $V\setminus (X\cup F)$ is mapped to a point. It follows that this transfer yields  the desired $\mathbb{Z}_2$-symmetric isostatic point-line framework in $\mathbb{R}^2$ with respect to $\theta$ and $\tau':\mathbb{Z}_2\to \mathcal{C}_2$. Since this process is clearly reversible, the result follows.
\end{proof}

\subsection{Combinatorial consequences}

As mentioned above, necessary conditions for a point-line framework with $\mathcal{C}_2$ or $\mathcal{C}_s$ symmetry to be isostatic in the plane were obtained in \cite{OP}. While this required methods from group representation theory, we may obtain some of these conditions more directly using    Corollaries~\ref{cor:mirhalffixed} and \ref{cor:mirhalffixed2} in conjunction with Corollary~\ref{cor:mirrortransfer} and Theorem~\ref{thm:bernd}. For example, if we start with an isostatic  point-line framework $(G,p,\ell)$  in the plane with $\mathcal{C}_2$ symmetry and with no point at the origin (the centre of the half-turn), then this can be transferred to an isostatic \emph{bar-joint} framework $(G,q)$ with $\mathcal{C}_s$ symmetry in the plane, and Theorem~\ref{thm:bernd} then implies that $G$ must be $(2,3)$-tight and that here exists exactly one edge of $G$ that is fixed by the non-trivial element in $\mathbb{Z}_2$. From the necessary conditions for  a point-line framework with $\mathcal{C}_2$ symmetry to be isostatic we can then also obtain necessary conditions for a point-line framework with $\mathcal{C}_s$ symmetry to be isostatic. 


We may also derive some new conditions. For example, if an isostatic point-line framework with $\mathcal{C}_s$ symmetry and underlying graph $G=(V_P\cup V_H, E)$  only has a single line, and $\theta$ acts freely on $V_P$,  then this line cannot lie along the mirror line, for otherwise this would transfer to an isostatic bar-joint framework with $\mathcal{C}_2$ symmetry in the plane with a vertex that is fixed by the half-turn,  contradicting Theorem~\ref{thm:bernd}.

We may even obtain both necessary and sufficient conditions for $\Gamma$-regular point-line frameworks to be isostatic in some special cases, as the following result shows. 

\begin{thm}  \label{thm:lamanptline2} Let $\mathbb{Z}_2=\langle \gamma \rangle$ and let $(G,p,\ell)$ be a $\mathbb{Z}_2$-regular point-line framework in $\mathbb{R}^2$ with respect to $\theta:\mathbb{Z}_2 \to \textrm{Aut}_{PH}(G)$ and $\tau:\mathbb{Z}_2 \to \mathcal{C}_2$. Suppose that $\gamma$ fixes each $i\in V_H$ and that $\theta$ acts freely on $V_P$. Then $(G,p,\ell)$ is isostatic if and only if  $G$ is $(2,3)$-tight and $|E_{\gamma}|=1$.
\end{thm}

\begin{proof} Suppose $(G,q',\ell')$ is isostatic with symmetry $\mathcal{C}_2$ as stated in the theorem. We transfer $(G,q',\ell')$ to a bar-joint framework $(G,p')$ with $\mathcal{C}_s$ symmetry as in Corollary~\ref{cor:mirhalffixed}. 
Then $(G,p')$ is also isostatic and it follows from Theorem~\ref{thm:bernd} that $G$ must be $(2,3)$-tight and $G$ must have exactly one edge that is fixed by $\gamma$. 

Conversely, suppose that $G$ is $(2,3)$-tight and $|E_{\gamma}|=1$. By the definition of $\theta$, we claim that if $(G,q',\ell')$ is $\mathbb{Z}_2$-regular, then so  is $(G,p')$. To see this, take an open neighbourhood of $\mathbb{Z}_2$-symmetric point-line realisations of $G$ in $\mathbb{R}^2$ (with respect to $\theta$ and $\tau:\mathbb{Z}_2\to \mathcal{C}_2$) around $(G,q',\ell')$ in which the rank of the corresponding point-line rigidity matrices is maintained. Then the transfer process described in the proofs of Theorem~\ref{thm:mirrorhtpair} and Corollary~\ref{cor:mirhalf} maps this neighbourhood to an open neighbourhood of $\mathbb{Z}_2$-symmetric bar-joint realisations of $G$ in $\mathbb{R}^2$ (with respect to $\theta$ and $\tau':\mathbb{Z}_2\to \mathcal{C}_s$) around $(G,p')$, since $\theta$ forces all vertices in $V_H$ to be mapped to points on the mirror line corresponding to the reflection in $\mathcal{C}_s$. Since the rank of the corresponding bar-joint rigidity matrices is maintained in this neighbourhood,  we may deduce that $(G,p')$ is $\mathbb{Z}_2$-regular, as claimed.

 Thus, it follows from Theorem~\ref{thm:bernd} that $(G,p')$ is isostatic.   Therefore, $(G,q',\ell')$ is also isostatic. 
\end{proof}

\subsection{Group pairings in Elliptic Geometry}\label{sec:elliptic}

To better fit our work into the historical evolution of rigidity theory, with a projective geometric background, we begin by recalling (and extending) the concept of static rigidity  -- the language of structural or civil engineers for several centuries.  

Let $R_\bS^d (G,p)$ denote the matrix of coefficients of the linear system describing spherical infinitesimal rigidity (see Equations (\ref{eq:inner_inf})-(\ref{eq:scale})).
We can define a spherical framework $(G,p)$ to be {\it statically rigid} if the row space of $R_\bS^d (G,p)$ spans  the space of all possible row vectors which are orthogonal to the space of trivial infinitesimal motions.   The elements of this space are called \emph{equilibrium loads} on the framework viewed as forces applied to each vertex, and static rigidity is the property that all equilibrium loads are linear combinations of the rows -- they are {\it resolved}. Thus static rigidity is based on the dimension of the row space being the difference between the number of columns and the dimension of the trivial motions of $\mathbb{S}^d$.   Since row rank equals column rank, this is the dual of infinitesimal rigidity and gives us an equivalent way of understanding infinitesimal motions.

Historically,  static rigidity was studied by engineers and was recognized as projectively invariant, first implicitly by M\"obius who wrote a textbook on statics using barycentric coordinates (his precursor of homogenous coordinates now used for projective geometry).  Balancing weighted points is the language of forces and statics.  Immediately after hearing a talk on the `new geometry' (projective geometry) in 1863, the British engineer Rankine (then writing a text on statics) published a short note observing the  invariance of statics under projective geometry \cite{Ran}.  This context of projective invariance was part of the milieu of Cayley and Klein when geometry was a shared vocabulary and approach of mathematicians and engineers. 

As part of the revival of the mathematical theory of rigidity in the 1970's, Crapo and Whiteley presented the  statics  of frameworks in terms of  explicitly projective notation  and reasoning, including references to 3D translations as rotations about lines at infinity (sliders) \cite{CrW}.  The work here builds on those continuing explorations. 

In this paper we have used the spherical model of frameworks with points on the equator to incorporate hyperplanes into Euclidean bar-joint frameworks, and to analyse the infinitesimal rigidity of this larger class of point-hyperplane frameworks.  This also implicitly incorporates projective transformations. Consider the following sequence of operations on a point-hyperplane framework in $\mathbb{R}^d$: first project to $\mathbb{S}^d$; then apply an isometry of $\mathbb{S}^d$; and finally reproject  to $\mathbb{R}^d$. The resulting point-hyperplane framework in $\mathbb{R}^d$ is a projective transformation of the original framework. By restricting this process to the upper hemisphere $\mathbb{S}^d_{>0}$, there is no ambiguity or collapsing of points in this process. Antipodal points on the equator, however, map to parallel hyperplanes (or the same hyperplane through the origin), but with opposite normals. 


When we consider certain $\Gamma$-symmetric spherical frameworks $(G,p)$,  where $\theta:\Gamma\to \textrm{Aut}(G)$ acts freely on the vertices, then we have seen in Section~\ref{sec:pairs} that we may invert the points corresponding to an index 2 subgroup of $\Gamma$ without changing the infinitesimal rigidity properties, in order to establish a group pairing $\tau(\Gamma) \leftrightarrow \tau'(\Gamma)$ as in Theorem~\ref{thm:pairs}. However, we have also seen in Section~\ref{subsec:nonfreepair} that if $\theta$ does not act freely on the vertices, then the presence of the fixed vertices implies that this partial inversion process destroys the symmetry of the spherical framework. Nevertheless, in the projection to Euclidean space, the symmetry can be recovered, as fixed points on the equator are mapped to fixed hyperplanes which may be shifted to go through the origin (recall Remark \ref{rem:shift}). Since antipodal points on the equator project to the same hyperplane, we can actually think of the symmetry as being present on the sphere as well, provided that we somehow identify antipodal points on the sphere.


This leads us back to the projective roots of infinitesimal rigidity, since the sphere with antipodal points identified is the  `metric projective space'  also called the elliptic geometry \cite{Wikipedia}.
This approach  is central to the approach in \cite{Con}, where the symmetries in projective space are described, and pairings of spherical symmetries are organized using inversions.  

We have explored these topics using two equivalent geometric models: the sphere with antipodal points identified  or equivalently, the sphere with points as equivalence classes of pairs of antipodal points.    The sphere with antipodal points identified  can be represented working from Lemma \ref{lem:inv} and the associated matrices.
  It is easy to see that $R_\bS^d (G,p)$ is rank equivalent to the following matrix which we call the basic spherical rigidity matrix:
\begin{displaymath} \bordermatrix{& & & & v_i & & & & v_j & & &
\cr & & & &  & & \vdots & &  & & &
\cr
 \{v_i,v_j\} & 0 & \ldots &  0 & p_j & 0 & \ldots & 0 &  p_i &  0 &  \ldots&  0
 \cr & & & &  & & \vdots & &  & & &
\cr \{v_i,v_i\} & 0 & \ldots &  0 & 2(p_i) & 0 & \ldots & 0 &  0 &  0 &  \ldots&  0
 \cr & & & &  & & \vdots & &  & & &
\cr v_i & 0 & \ldots &  0 &p_{i}& 0 & \ldots & 0 & 0&  0 &  \ldots&  0
\cr & & & &  & & \vdots & &  & & &
\cr  v_j & 0 & \ldots &  0 & 0& 0 & \ldots & 0 & p_{j} &  0 &  \ldots&  0
\cr & & & &  & & \vdots & &  & & &
\cr & & & &  & & \vdots & &  & & &
}
\textrm{,}\end{displaymath}

We can use inversion on this matrix to represent the geometry of this
more complete projective model of equivalent frameworks in the projective or elliptic space.
Specifically, by applying inversion to any chosen subset of the vertices, taking $p_i$ to $\epsilon_i p_i$  with $\epsilon_i = \pm 1$, the matrix can be transformed by row and column multiplications into the following form: 
$R_\bS^d (G,\epsilon p)$:

\begin{displaymath} \bordermatrix{& & & &  (\epsilon_i) v_i & & & &  (\epsilon_j) v_j & & &
\cr & & & &  & & \vdots & &  & & &
\cr
  \epsilon_i\epsilon_j \{ v_i, v_j\} & 0 & \ldots &  0 & (\epsilon_j) p_j & 0 & \ldots & 0 &   (\epsilon_i)p_i &  0 &  \ldots&  0
 \cr & & & &  & & \vdots & &  & & &
\cr \{v_i,v_i\} & 0 & \ldots &  0 & 2((\epsilon_i)p_i) & 0 & \ldots & 0 &  0 &  0 &  \ldots&  0
 \cr & & & &  & & \vdots & &  & & &
\cr   v_i & 0 & \ldots &  0 &(\epsilon_i)p_{i}& 0 & \ldots & 0 & 0&  0 &  \ldots&  0
\cr & & & &  & & \vdots & &  & & &
\cr  v_j & 0 & \ldots &  0 & 0& 0 & \ldots & 0 &  (\epsilon_j)p_{j} &  0 &  \ldots&  0
\cr & & & &  & & \vdots & &  & & &
\cr & & & &  & & \vdots & &  & & &
}
\textrm{,}\end{displaymath}

Lemma \ref{lem:inv}  implies that all the matrix properties (rank, size, dimensions of kernel and co-kernel) are preserved.  Therefore the associated rigidity properties of infinitesimal and static rigidity are equivalent.  If we collect the class of all these inversions, we create the {\it  $\iota$-equivalent spherical frameworks}.  All $\iota$-equivalent spherical frameworks have the same projection into $\mathbb{R}^d$ given that we do not distinguish positively and negatively weighted points -- or hyperplanes through the origins with opposite normals.  A symmetry of the frameworks modulo this $\iota$-equivalence is a transformation of equivalence classes.  

Under this identification, half-turn symmetry and mirror symmetry are $\iota$-equivalent -- with the added clarification that: where a mirror symmetry appears to take  a point on the equator perpendicular to the mirror to its antipode, this is a fixed point.   In this perspective, a half-turn not only fixes the center of rotation, but all points on the equator, in all dimensions.  In projection, we must identify both versions of a line ($\pm $ the normal) and see the line as fixed, as we also see $\pm $ a point as the same fixed point.  

The static theory outlined at the beginning of this subsection extends indirectly to point-hyperplane frameworks \cite{Elfakhari}.   
This static theory provides the basis for the theory of tensegrity frameworks \cite{Elfakhari}, but here there are subtleties in a projective theory of tensegrity frameworks on the sphere, and in their projections, which deserve an extended exploration.   The static theory can also be applied to the row space and row dependencies of orbit matrices under symmetry \cite{SWorbit}. 

\section{Further work}

\noindent 1. \emph{Global rigidity.} Connelly and Whiteley \cite{ConWhi} explored the connections between  global rigidity of frameworks in spherical space and their projections to Euclidean Space. The key technique used was to model spherical frameworks as `cone frameworks' in Euclidean space.  
 Such a framework has a cone vertex realised at the origin which is adjacent to all other vertices (recall also Section~\ref{subsec:sph}).
Observing that inversion within a cone preserved global rigidity, we anticipate a number of the results here will transfer.   Since equilibrium stresses are a second tool for global rigidity, and we can trace the impact of inversion and projection on the signs of the stresses, the tools exist for a more detailed analysis of the transfer and the pairings to track the effect on global rigidity \cite{TWHandG, BSWWhb}. However this is a largely unexplored problem in the presence of symmetry, or indeed for point-hyperplane frameworks so we leave this as future work.\\


\noindent 2. \emph{Change of metrics.} Infinitesimal rigidity, as a projective invariant, is invariant under change of metrics among those with a shared projective geometry \cite{NW,SaliolaWh}.  With this background, and the recognition that hyperbolic frameworks (as cones) project to Minkowski frameworks, we anticipate that the results for the Euclidean and spherical spaces extend to Minkowski space (or more generally any pseudo-Euclidean space) and then to the hyperbolic and De Sitter spaces. See \cite{NW,BSWWSphere} for more details. We further expect that the pairing results of this paper can be adapted to this context, when the corresponding symmetries exist.  In particular, half-turn symmetry will correspond to mirror symmetry by the known transfers of rigidity results from the Euclidean space to Minkowski space.  Since Minkowski space has the full space of translations, we anticipate that there are full extensions to a theory of point-hyperplane frameworks in Minkowski space.\\ 

\noindent 3.  \emph{Parallel drawings.} It is well known that, for the plane, the theory of parallel drawings is isomorphic to the theory of infinitesimal rigidity \cite{SWHandR} -- so the pairing of half-turn symmetry with mirror symmetry in the plane also transfers.  More generally, the theory of parallel drawings of point-hyperplane frameworks in all dimensions is projectively invariant.  This suggests that pairings of symmetries will have analogs for the theory of symmetric parallel drawings.  





\section*{Acknowledgements}  

Our collaboration on this paper was initiated during the BIRS workshop on `Advances in Combinatorial and Geometric Rigidity' in July 2015.    
Walter Whiteley's work has been supported by grants from NSERC (Canada). Katie Clinch's work is supported by JST CREST (JPMJCR14D2).


\bibliographystyle{abbrv}

\begin{thebibliography}{30}
\bibitem{Abb} T. Abbott, Generalizations of Kempe's universality theorem, Master's thesis, Massachusetts Institute of Technology (2008).

\bibitem{AR} L. Asimow and B. Roth, {\em The rigidity of graphs}, Transactions of the American Mathematical Society, 245 (1978), 279--289.

\bibitem{Bis} D. M. Bishop, Group Theory and Chemistry, Clarendon Press, Oxford, 1973.

\bibitem{Cetal} R. Connelly, P. Fowler, S. Guest, B. Schulze and W. Whiteley, {\em When is a symmetric pin-jointed framework isostatic?}, International Journal of Solids and Structures, 46 (2009) 762--773.

\bibitem{ConWhi}  R. Connelly and W.J. Whiteley: {\em Global Rigidity: the Effect of Coning},
Discrete and Computational Geometry, 43  (2010),  717--735.

\bibitem{Con} J. Conway and D. Smith, On Quaternions and Octonions, A K Peters/CRC Press, 2003.

\bibitem{CrW} H. Crapo and W. Whiteley, \emph{Statics of Frameworks and Motions of Panel Structures}, A Projective Geometric Introduction, Structural Topology, 6 (1982) 43--82.

\bibitem{Elfakhari} Y. Eftekhari,  Geometry of Point-Hyperplane and Spherical Frameworks, Ph.D. Thesis, York University, 2017.  

\bibitem{EJNSTW} Y. Eftekhari, B. Jackson, A. Nixon, B. Schulze, S. Tanigawa and W. Whiteley, {\em Point-hyperplane frameworks, slider joints, and rigidity preserving transformations}, Journal of Combinatorial Theory: Series B, 135 (2019) 44--74.

\bibitem{FJK} Zs. Fekete, T. Jord\'an and V. E. Kaszanitzky, {\em Rigid two-dimensional frameworks with two coincident points}, Graphs and Combinatorics, 31:3 (2014) 585--599.

\bibitem{guestfow} S.D.~Guest and P.W.~Fowler, \emph{Symmetry conditions and finite mechanisms},
Mechanics of Materials and Structures, 2:2 (2007), 293--301.

\bibitem{GSW} S. Guest, B. Schulze and W. Whiteley, {\em When is a symmetric body-bar structure isostatic?}, International Journal of Solids and Structures, 47 (2010) 2745--2754.

\bibitem{Ikeshita} R.~Ikeshita, {\em  Infinitesimal rigidity of symmetric frameworks}, Master's thesis, Department of Mathematical Informatics, University of Tokyo, 2015.

\bibitem{IkTan} R.~Ikeshita and S. Tanigawa, {\em  Count matroids of group-labeled graphs}, Combinatorica, 38:5 (2018) 1101--1127.

\bibitem{Izmestiev} I.~Izmestiev, {\em  Projective background of the infinitesimal rigidity of frameworks},
 Geometriae Dedicata, 140 (2009) 183--203.

\bibitem{JO}
B.~Jackson and J.~Owen, {\em A characterisation of the generic rigidity of 2-dimensional point-line frameworks}, Journal of Combinatorial Theory: Series B, 119 (2016) 96--121.

\bibitem{jkt} T. Jord\'{a}n, V. Kaszanitzky and S. Tanigawa, \emph{Gain-sparsity and symmetry-forced rigidity in the plane}, Discrete and Computational Geometry, 55 (2016) 314--372.

\bibitem{kangguest} R.D.~Kangwai and S.D.~Guest, \emph{Detection of finite mechanisms in symmetric structures}, International Journal of Solids and Structures, 36 (1999), 5507--5527.

 \bibitem{mt1} J.~Malestein and L.~Theran, \emph{Frameworks with forced symmetry I: reflections and rotations}, Discrete and Computational Geometry, 54:2 (2014) 339--367.

\bibitem{NS} A.~Nixon and B.~Schulze, {\em Symmetry-forced rigidity of frameworks on surfaces},  Geometriae Dedicata, 182:1 (2016) 163--201.

\bibitem{NW} A. Nixon and W. Whiteley, {\em Change of metrics in rigidity theory}, in Handbook of Geometric Constraint Systems Principles, J.~Sidman, M.~Sitharam, A.~St. John, editors, 
Chapman \& Hall CRC, 2018.

\bibitem{OP} J. Owen and S. Power, {\em Frameworks symmetry and rigidity}, International Journal of Computational Geometry and Applications, 20:6 (2010) 723--750.

\bibitem{Pog} A.V. Pogorelov, {\em Extrinsic geometry of convex surfaces}, Translation of the
1969 edition, Translations of Mathematical Monographs 35, AMS, 1973.

\bibitem{Ran} W. J. M. Rankine, \emph{On the Application of Barycentric Perspective to the Transformation
of Structures},  Phil. Mag. Series, 4:26 (1863), 387--388.

\bibitem{RosenWills} N. Rosenauer and A.H. Willis, Kinematics of Mechanisms, Dover Publications, New
York 1967.  

\bibitem{SaliolaWh} F. Saliola and W. Whiteley, {\em Some notes on the equivalence of first-order rigidity in various geometries}, Preprint (2007) arXiv:0709.3354.

\bibitem{BSfinite} B. Schulze, {\em Symmetry as a sufficient condition for a finite flex}, SIAM Journal on Discrete Mathematics, 24:4 (2010) 1291--1312.

\bibitem{SchC3} B. Schulze, {\em Symmetric versions of Laman's Theorem}, Discrete and Computational Geometry, 44:4 (2010) 946-972.

\bibitem{Sch} B. Schulze, {\em Symmetric Laman theorems for the groups $C_2$ and $C_s$}, 
Electronic Journal of Combinatorics, 17:1 (2010) 1--61.

\bibitem{BShb}
B. Schulze, {\em Combinatorial rigidity of frameworks under symmetry}, in Handbook of Geometric Constraint Systems Principles, J.~Sidman, M.~Sitharam, A.~St. John, editors, 
Chapman \& Hall CRC, 2018.

\bibitem{STbb} B. Schulze and S. Tanigawa, {\em Linking rigid bodies symmetrically}, European Journal of Combinatorics, 42 (2014) 145--166.

\bibitem{ST} B. Schulze and S. Tanigawa, {\em Infinitesimal rigidity of symmetric bar-joint frameworks}, SIAM Journal on Discrete Mathematics, 29:3 (2015) 1259--1286.

\bibitem{SWorbit} B. Schulze and W. Whiteley, {\em The orbit rigidity matrix of a symmetric framework}, Discrete and Computational Geometry, 46:3 (2011) 561--598.

\bibitem{BSWWSphere}
B.~Schulze and W.~Whiteley, {\em Coning, symmetry and spherical frameworks}, Discrete and Computational Geometry, 48 (2012) 622--657.

\bibitem{SWHandR} B. Schulze and W. Whiteley, {\em Rigidity and Scene Analysis},  in Handbook of Discrete and
Computational Geometry, C.D. Toth, J. O’Rourke, J.E. Goodman, editors, Third Edition,
Chapman \& Hall CRC, 2018.

\bibitem{BSWWhb}
B. Schulze and W. Whiteley, {\em Symmetry and Rigidity}, in Handbook of Discrete and
Computational Geometry, C.D. Toth, J. O’Rourke, J.E. Goodman, editors, Third Edition,
Chapman \& Hall CRC, 2018.

\bibitem{Tan15} S. Tanigawa, {\em Matroids of gain graphs in applied discrete geometry}, Transactions AMS, 367 (2015) 8597--8641.

\bibitem{TWHandG} T. Jord\'an and W. Whiteley,  {\em Global Rigidity} in Handbook of Discrete and Computational Geometry, C.D. Toth, J. O'Rourke, J.E. Goodman, editors, Third Edition, 
Chapman \& Hall CRC, 2018.

\bibitem{Wikipedia} Wikipedia,  {\em Elliptic Geometry}, \url{https://en.wikipedia.org/wiki/Elliptic_geometry}.

\end{thebibliography}
\def\lfhook#1{\setbox0=\hbox{#1}{\ooalign{\hidewidth
  \lower1.5ex\hbox{'}\hidewidth\crcr\unhbox0}}}

\end{document}